\PassOptionsToPackage{unicode=true}{hyperref} %
\PassOptionsToPackage{hyphens}{url}
\documentclass[]{article}
\usepackage{lmodern}
\usepackage{amssymb,amsmath}
\usepackage{ifxetex,ifluatex}
\usepackage{fixltx2e} %
\ifnum 0\ifxetex 1\fi\ifluatex 1\fi=0 %
  \usepackage[T1]{fontenc}
  \usepackage[utf8]{inputenc}
  \usepackage{textcomp} %
\else %
  \usepackage{unicode-math}
  \defaultfontfeatures{Ligatures=TeX,Scale=MatchLowercase}
\fi
\IfFileExists{upquote.sty}{\usepackage{upquote}}{}
\IfFileExists{microtype.sty}{%
\usepackage[]{microtype}
\UseMicrotypeSet[protrusion]{basicmath} %
}{}
\IfFileExists{parskip.sty}{%
\usepackage{parskip}
}{%
\setlength{\parindent}{0pt}
\setlength{\parskip}{6pt plus 2pt minus 1pt}
}
\usepackage{hyperref}
\hypersetup{
            pdftitle={Recursive, parameter-free, explicitly defined interpolation nodes for simplices},
            pdfauthor={Tobin Isaac},
            pdfborder={0 0 0},
            breaklinks=true}
\usepackage{color}
\usepackage{fancyvrb}

\DefineVerbatimEnvironment{Highlighting}{Verbatim}{commandchars=\\\{\}}
\usepackage{framed}
\definecolor{shadecolor}{RGB}{248,248,248}
\newenvironment{Shaded}{\begin{snugshade}}{\end{snugshade}}

\newcommand{\BuiltInTok}[1]{#1}

\newcommand{\CommentTok}[1]{\textcolor[rgb]{0.56,0.35,0.01}{\textit{#1}}}

\newcommand{\ControlFlowTok}[1]{\textcolor[rgb]{0.13,0.29,0.53}{\textbf{#1}}}

\newcommand{\DecValTok}[1]{\textcolor[rgb]{0.00,0.00,0.81}{#1}}

\newcommand{\FloatTok}[1]{\textcolor[rgb]{0.00,0.00,0.81}{#1}}

\newcommand{\ImportTok}[1]{#1}

\newcommand{\KeywordTok}[1]{\textcolor[rgb]{0.13,0.29,0.53}{\textbf{#1}}}
\newcommand{\NormalTok}[1]{#1}
\newcommand{\OperatorTok}[1]{\textcolor[rgb]{0.81,0.36,0.00}{\textbf{#1}}}

\usepackage{longtable,booktabs}
\IfFileExists{footnote.sty}{\usepackage{footnote}\makesavenoteenv{longtable}}{}
\usepackage{graphicx,grffile}
\makeatletter
\def\maxwidth{\ifdim\Gin@nat@width>\linewidth\linewidth\else\Gin@nat@width\fi}
\def\maxheight{\ifdim\Gin@nat@height>\textheight\textheight\else\Gin@nat@height\fi}
\makeatother
\setkeys{Gin}{width=\maxwidth,height=\maxheight,keepaspectratio}
\providecommand{\tightlist}{%
  \setlength{\itemsep}{0pt}\setlength{\parskip}{0pt}}
\ifx\paragraph\undefined\else
\let\oldparagraph\paragraph
\renewcommand{\paragraph}[1]{\oldparagraph{#1}\mbox{}}
\fi
\ifx\subparagraph\undefined\else
\let\oldsubparagraph\subparagraph
\renewcommand{\subparagraph}[1]{\oldsubparagraph{#1}\mbox{}}
\fi

\makeatletter
\def\fps@figure{htbp}
\makeatother

\usepackage{amsthm}

\newtheorem{proposition}{Proposition}[section]

\theoremstyle{definition}

\theoremstyle{definition}

\theoremstyle{definition}

\theoremstyle{remark}

\newcommand{\bb}{\boldsymbol{b}}
\newcommand{\xx}{\boldsymbol{x}}
\newcommand{\XX}{\boldsymbol{X}}
\newcommand{\aal}{\boldsymbol{\alpha}}
\newcommand{\RR}{\boldsymbol{R}}
\newcommand{\bs}{\backslash}
\usepackage{tikz}
\usepackage[style=alphabetic,maxbibnames=999,giveninits=true,date=year]{biblatex}
\title{Recursive, parameter-free, explicitly defined interpolation nodes for simplices}
\author{Tobin Isaac\footnote{\href{mailto:tisaac@cc.gatech.edu}{\nolinkurl{tisaac@cc.gatech.edu}}, School of Computational Science and Engineering, Georgia Tech.}}
\date{}

\begin{document}
\maketitle
\begin{abstract}
A rule for constructing interpolation nodes for \(n\)th degree polynomials on
the simplex is presented. These nodes are simple to define recursively from
families of 1D node sets, such as the Lobatto-Gauss-Legendre (LGL) nodes.
The resulting nodes have attractive properties: they are fully symmetric,
they match the 1D family used in construction on the edges of the simplex,
and the nodes constructed for the \((d-1)\)-simplex are the boundary traces of
the nodes constructed for the \(d\)-simplex. When compared using the Lebesgue
constant to other explicit rules for defining interpolation nodes, the nodes
recursively constructed from LGL nodes are nearly as good as the \emph{warp \&
blend} nodes of \textcite{Warb06} in 2D (which, though defined differently, are very
similar), and in 3D are better than other known explicit rules by increasing
margins for \(n > 6\). By that same measure, these recursively defined nodes
are not as good as implicitly defined nodes found by optimizing the Lebesgue
constant or related functions, but such optimal node sets have yet to be
computed for the tetrahedron. A reference python implementation has been
distributed as the \texttt{recursivenodes} package, but the simplicity of the
recursive construction makes them easy to implement.
\end{abstract}

\renewbibmacro*{doi+eprint+url}{%
  \iftoggle{bbx:doi}
    {\printfield{doi}}
    {}%
\newunit\newblock
  \iftoggle{bbx:eprint}
    {\usebibmacro{eprint}}
    {}%
\newunit\newblock
  \iftoggle{bbx:url}
        {%
      \iffieldundef{doi}
{\usebibmacro{url+urldate}}
        {%
          \clearfield{url}%
        }%
    }
    {}}

\hypertarget{definition-of-the-recursive-rule}{%
\section{Definition of the recursive rule}\label{definition-of-the-recursive-rule}}

The motivating example for this work is the use of Lagrange polynomials as
shape functions for the finite element approximation space
\(\mathcal{P}_n(\Delta^d)\): polynomials of degree at most \(n\) on the
\(d\)-simplex. A Lagrange polynomial basis \(\Phi_X = \{\varphi_{i}\} \subset \mathcal{P}_n(\Delta^d)\) is defined by a set of \emph{interpolation nodes} \(X = \{\xx_{i}\} \subset \Delta^d\) as \(\Phi_X = \{ \varphi_{i} \in \mathcal{P}_n(\Delta^d): \varphi_{i}(\xx_{j}) = \delta_{ij} \}\). While some of the properties of an implementation of the finite
element method depend only on the approximation space, the basis used, whether Lagrange or other, can
affect the convergence, numerical stability, and computational efficiency of the method.
Convergence is affected by the way the basis is used to approximate data,
numerical stability by presence of round-off errors, in the construction of both the basis and the resulting systems of equations, and computational efficiency by the complexity of common tasks like applying a mass or derivative matrix.
Discussion of each of these aspects follows the definition of the
interpolation nodes that are the main contribution of this work.

The nodes are dimensionally recursive, building from points on the interval
\([0,1]\). A \emph{1D node set} is a set of points \(X_n = \{x_{n,i}\}_{i=0}^n \subset [0,1]\) that is increasing and symmetric about \(1/2\), \(x_{n,i} = 1 - x_{n,n-i}\). A \emph{1D node family} is a collection \(\XX = \{X_n\}_{n\in \mathbb{N}_0}\). Examples include equispaced nodes, symmetric
Gauss-Jacobi quadrature nodes, and symmetric Lobatto-Gauss-Jacobi quadrature
nodes.

The new nodes are naturally defined on the barycentric \(d\)-simplex,
\[\Delta^d_{\text{bary}} = \{ \bb = (b_0, \dots, b_d) \in
\mathbb{R}^{d+1}_+ : {\textstyle \sum_i}\, b_i = 1\},\]
and are naturally indexed by the multi-indices
\[A^d_n = \{ \aal = (\alpha_0, \dots, \alpha_d) \in
\mathbb{N}_0^{d+1} : |\aal| = n \}.\]
This work uses the standard notation \(|\aal| = \sum_i \alpha_i\), and further defines:

\begin{itemize}
\tightlist
\item
  \(\#\xx\) as the length of a tuple (multi-index or vector),
\item
  \(\xx_{\bs i}\) as the tuple formed by
  removing the \(i\)th element, and
\item
  \(\xx_{+i}\) as the augmentation of a tuple by inserting a zero for
  the \(i\)th element.
\end{itemize}

Given a 1D node family \(\XX\), the recursive definition of the interpolation node
\(\bb_{\XX}(\aal) \in \Delta^{\#\aal-1}_{\text{bary}}\)
is
\begin{equation}
\label{eq:recursive}
\bb_{\XX}(\aal) =
\begin{cases}
(1), & \#\aal = 1, \\
\frac{\displaystyle {\textstyle \sum_i}\, x_{|\aal|,|\aal_{\bs i}|}
\bb_{\XX}(\aal_{\bs i})_{+i}}
{\displaystyle {\textstyle \sum_i}\, x_{|\aal|,|\aal_{\bs i}|}}, & \#\aal > 1.
\end{cases}
\end{equation}
The full \(d\)-simplex node set is
\begin{equation}
R^d_{\XX,n} = \{\bb_{\XX}(\aal): \aal \in A^d_n\},
\end{equation}
and the full \(d\)-simplex node family is
\begin{equation}
\RR^d_{\XX} = \{R^d_{\XX,n}\}_{n\in\mathbb{N}}.
\end{equation}
Unless otherwise specified, the 1D node family \(\XX\) is taken to be
the Lobatto-Gauss-Legendre (LGL) family
\(\XX_{\text{LGL}}\). Some examples are illustrated in Fig.~\ref{fig:illus}.

\begin{figure}
\centering
\includegraphics{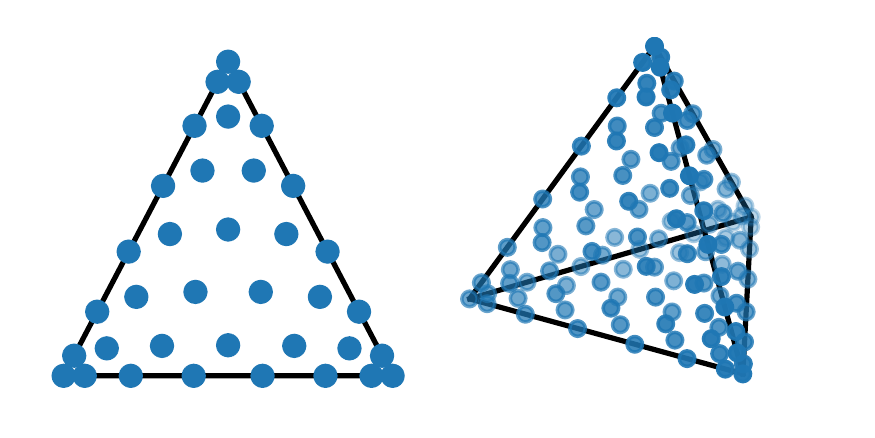}
\caption{\label{fig:illus}The nodes \(R^2_{\XX,7}\) and \(R^3_{\XX,7}\), mapped to equilateral simplices.}
\end{figure}

\hypertarget{intuition}{%
\section{Intuition behind the recursive rule}\label{intuition}}

\textcite{BlPo06} observed that Fekete points for the triangle, which have good
interpolation properties, points in the interior nearly project onto LGL nodes on
the edges of the triangle:

\begin{quote}
Some intriguing observations can be made regarding the location of some of
the Fekete points in a given set. {[}\ldots{}{]} If an imaginary line is drawn
through the nodes {[}\ldots{}{]} the two Fekete nodes sit on this line, close to the
two zeros of the second Lobatto polynomial, Lo2, scaled by the length of the
imaginary line.
\end{quote}

From this comes the idea that, if a good family of interpolation nodes on the
\((d-1)\)-simplex are already known, a heuristic for locating the node
\(\bb_{\XX}(\aal)\) in the \(d\)-simplex is to choose a point whose
projection onto each facet from the opposite vertex is one of those good nodes,
\begin{equation}
\frac{\bb_{\XX}(\aal)_{\bs i}}{1 - b_{\XX,i}(\aal)} = \bb_{\XX}(\aal_{\bs i}), \quad \forall\ i\in \{0, \dots, d\}.
\label{eq:overdetermined}
\end{equation}

Unfortunately, this is an overdetermined set of requirements.

Consider for example the placement of the interpolation node with multi-index
\(\aal = (1,2,3)\) in the barycentric triangle (this is one of the nodes
for \(n=|\aal|=6\)). The LGL nodes are good interpolation nodes, so the desire is for \(\bb_{\XX}((1,2,3))\) to project onto the LGL nodes
associated with the multi-indices \((2,3)\) (one of the nodes for \(n=2+3=5\)),
\((1,3)\) (\(n=4\)), and \((1,2)\) (\(n=3\)), as illustrated in Fig.~\ref{fig:lglproj} (right).
The projection lines nearly intersect at one
point, but not quite. The system
\eqref{eq:overdetermined} has a solution if \(\XX\) is the family of equispaced nodes, and the
solution is an equispaced node in the triangle, as seen in Fig.~\ref{fig:lglproj} (left).

\begin{figure}
\centering
\includegraphics{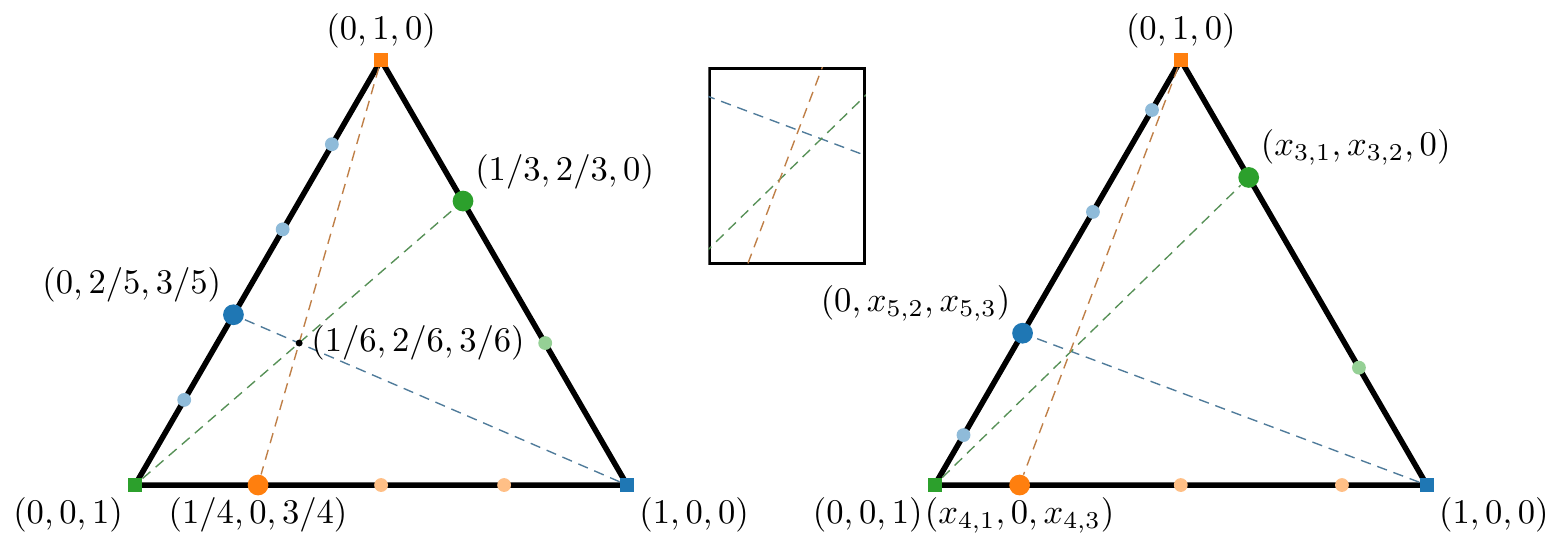}
\caption{\label{fig:lglproj}Desired projections for \(\bb_{\XX}((1,2,3))\). If \(\XX\) is the equispaced 1D node family (left), the projections meet at an equispaced node in the triangle. If \(\XX\) is the LGL 1D node family (right), the lines do not actually meet at one point (detail, 32x magnified).}
\end{figure}

Any point in the interior of the triangle is in the convex hull of its
projections onto the edges, so if a node location does satisfy
\eqref{eq:overdetermined}, then it can be expressed as a barycentric
combination of its projections. The equispaced nodes of the triangle not only have projections that are equispaced nodes on the edges, but their barycentric weights have a remarkable property.\\
\begin{proposition}
\protect\hypertarget{prp:bary}{}{\label{prp:bary} }Let the barycentric coordinates of the equispaced node associated with
\(\aal = (\alpha_0, \alpha_1, \alpha_2)\), \(|\aal| = n\), be
\(\bb = (\alpha_0 / n, \alpha_1 / n, \alpha_2 / n)\). Let its projections onto the edges be
\[
\begin{aligned}
\bb_0 &= (0, \alpha_1 / (n - \alpha_0), \alpha_2 / (n - \alpha_0)),\\
\bb_1 &= (\alpha_0 / (n - \alpha_1), 0, \alpha_2 / (n - \alpha_1)), \\
\bb_2 &= (\alpha_0 / (n - \alpha_2), \alpha_1 / (n - \alpha_2), 0).
\end{aligned}
\]
Then \(\bb\) is a convex combination of \(\bb_0\), \(\bb_1\), and \(\bb_2\) with
(unnormalized) barycentric weights
\((1 - \alpha_0 / n) : (1 - \alpha_1 / n) : (1 - \alpha_2 / n)\).
\end{proposition}

In Proposition \ref{prp:bary}, the barycentric weights describing equispaced
nodes in the triangle are themselves 1D equispaced nodes on the edge (see
Fig.~\ref{fig:equibary}, left).
By analogy, a heuristic for approximating a solution to the overdetermined system \eqref{eq:overdetermined} is to use the same 1D node family \(\XX\)
that was used for the projection points as barycentric weights for combining them (see Fig.~\ref{fig:equibary}, right), which restates \eqref{eq:recursive}.

\begin{figure}
\centering
\includegraphics{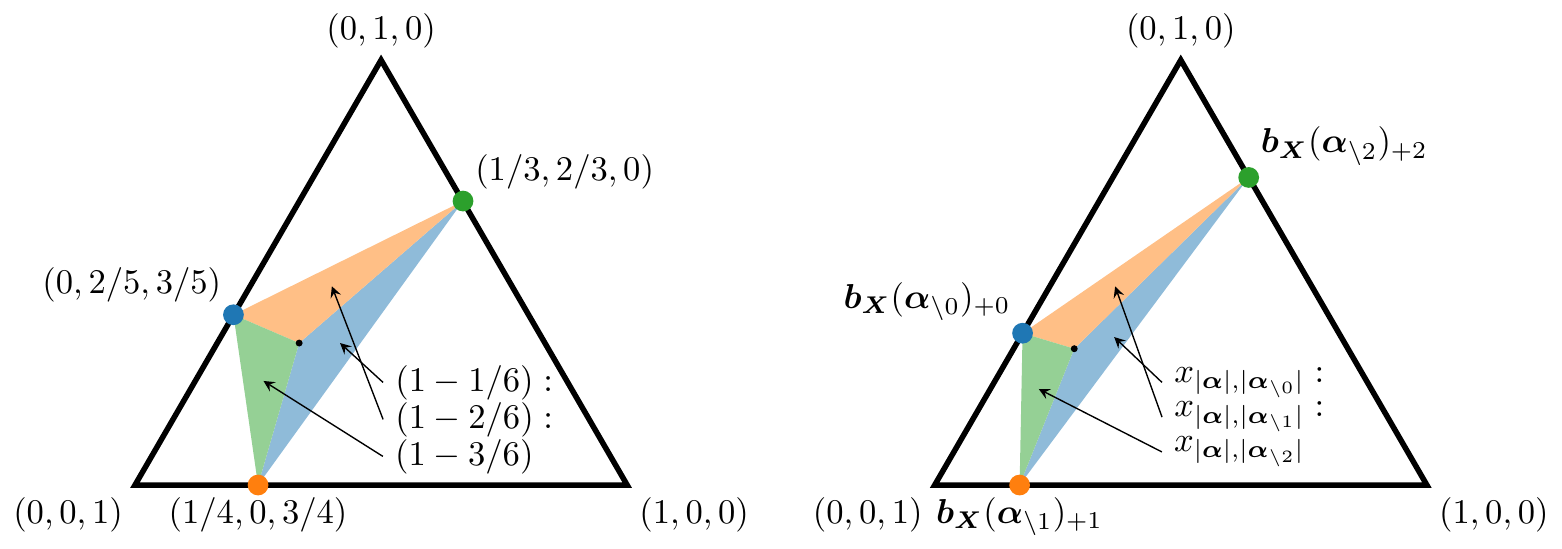}
\caption{\label{fig:equibary}Defining \(\bb_{\XX}((1,2,3))\) by barycentric coordinates relative to the projection points. If \(\XX\) is the equispaced 1D node family (left), it is the same as the point as the intersection of projection lines (see Fig.~\ref{fig:lglproj}). If \(\XX\) is an arbitrary 1D node family (right), it is the recursive rule \eqref{eq:recursive}.}
\end{figure}

\hypertarget{comparison}{%
\section{Comparison to other node families}\label{comparison}}

The recursive rule \eqref{eq:recursive} generates node families
\(\RR^d_{\XX}\) for the \(d\)-simplex in each dimension. This
section compares them to other node families with respect to several metrics
that are relevant to finite element computations.

\hypertarget{symmetry}{%
\subsection{Boundary and symmetry properties}\label{symmetry}}

The nodes \(R^{d}_{\XX,n}\) have three non-numerical properties
that make them convenient to use when implementing the finite element method.

\begin{enumerate}
\def\labelenumi{\Roman{enumi}.}
\item
  \textbf{Symmetry:} The symmetry group of the \(d\)-simplex is the group \(S_{d+1}\):
  for \(\Delta^{d}_{\text{bary}}\), each symmetry corresponds to a permutation of
  the coordinates. It is clear that the recursive rule \eqref{eq:recursive}
  respects these symmetries, that \(\bb_{\XX} (\sigma (\aal)) = \sigma(\bb_{\XX} (\aal))\). This is useful when a \(d\)-simplex is viewed from
  multiple orientations, such as when it is the interface between cells.
\item
  \textbf{Equivalence to \(\XX\) when \(d=1\):} The node sets in
  \(\XX\) must be symmetric about \(1/2\), so if \(\aal = (\alpha_0, \alpha_1)\) then \(\sum_i x_{|\aal|,|\aal_{\bs i}|} = x_{|\aal|,\alpha_1} + x_{|\aal|,\alpha_0} = 1.\)
  The recursive rule \eqref{eq:recursive} then becomes
  \[
  \bb_{\XX}(\aal) =
  x_{|\aal|,\alpha_1} (0,1) + x_{|\aal|,\alpha_0}
  (1,0) =
  (x_{|\aal|,\alpha_0},x_{|\aal|,\alpha_1}).
  \]
  In other words, \(R^1_{\XX,n}\) is the 1D node set \(X_n\) mapped to the
  barycentric line \(\Delta^1_{\text{bary}}\).
\item
  \textbf{Recursive boundary traces:} Problems solved by the finite element
  method can have forms computed over surfaces: data for Neumann boundary
  conditions or jump terms in discontinuous Galerkin methods, for example. A
  good node set should induce good shape functions for \(\mathcal{P}_n(\Delta^d)\),
  but also for the trace spaces on the boundary facets, which are embeddings of
  \(\mathcal{P}_n(\Delta^{d-1})\).
\end{enumerate}

The following proposition show that if the 1D node family
\(\XX\) has nodes at the endpoints, then
\(R^{d}_{\XX,n}\) has nodes on each boundary facet of
\(\Delta^{d}_{\text{bary}}\) that are the \(R^{d-1}_{\XX,n}\)
nodes mapped onto that facet, and so they are appropriate for defining
Lagrange polynomials on the trace space.\\
\begin{proposition}
\protect\hypertarget{prp:trace}{}{\label{prp:trace} }Let \(\XX\) be a 1D node family such that
\(x_{n,0} = 0\) and \(x_{n,n} = 1\) for all \(n \geq 1\). Let \(\aal\)
be a multi-index such that \(|\aal| \geq 1\), \(\#\aal > 1\), and \(\alpha_j = 0\).
Then \(\bb_{\XX}(\aal) = \bb_{\XX}(\aal_{\bs j})_{+j}\).
\end{proposition}

\begin{proof}
{}If \(\alpha_j = 0\), then \(|\aal_{\bs j}| = |\aal|\),
so \(x_{|\aal|,|\aal_{\bs j}|} = 1\). Therefore,
\begin{align}
\bb_{\XX}(\aal) &=
\frac{\displaystyle {\textstyle \sum_i}\, x_{|\aal|,|\aal_{\bs i}|}
\bb_{\XX}(\aal_{\bs i})_{+i}}
{\displaystyle {\textstyle \sum_i}\, x_{|\aal|,|\aal_{\bs i}|}} \nonumber \\
&=
\frac{\displaystyle x_{|\aal|,|\aal_{\bs j}|}
\bb_{\XX}(\aal_{\bs j})_{+j} + {\textstyle \sum_{i \neq j}}\, x_{|\aal|,|\aal_{\bs i}|}
\bb_{\XX}(\aal_{\bs i})_{+i}}
{\displaystyle x_{|\aal|,|\aal_{\bs j}|} + {\textstyle \sum_{i \neq j}}\, x_{|\aal|,|\aal_{\bs i}|}} \nonumber \\
&=
\frac{\displaystyle \bb_{\XX}(\aal_{\bs j})_{+j} + {\textstyle \sum_{i \neq j}}\, x_{|\aal|,|\aal_{\bs i}|}
\bb_{\XX}(\aal_{\bs i})_{+i}}
{\displaystyle 1 + {\textstyle \sum_{i \neq j}}\, x_{|\aal|,|\aal_{\bs i}|}}.
\label{eq:expansion}
\end{align}
If \(\#\aal = 2\), then
\(x_{|\aal|,|\aal_{\bs i}|} = x_{\alpha_i,0} = 0\) for \(i \neq j\). Then \eqref{eq:expansion} simplifies to
\(\bb_{\XX}(\aal_{\bs j})_{+j}\). This proves the base case.

Now assume the property holds if \(\#\aal \leq d\), and let
\(\#\aal = d + 1\). If \(i \neq j\), then \(\aal_{\bs i}\)
has a zero at an index \(\hat\jmath\in\{j,j-1\}\), so
\[
\bb_{\XX}(\aal_{\bs i})_{+i}
=
\bb_{\XX}(\aal_{\bs i \bs \hat\jmath})_{+\hat\jmath +i}.
\]
The order can be switched: there is \(\hat\imath\in\{i,i-1\}\) such that
\[
\bb_{\XX}(\aal_{\bs i \bs \hat\jmath})_{+\hat\jmath +i}
=
\bb_{\XX}(\aal_{\bs j \bs \hat\imath})_{+\hat\imath +j}.
\]
So by relabeling and using \eqref{eq:recursive} this implies
\[\begin{aligned}
{\textstyle \sum_{i \neq j}}\, x_{|\aal|,|\aal_{\bs i}|}
\bb_{\XX}(\aal_{\bs i})_{+i} &= 
{\textstyle \sum_{\hat\imath}}\, x_{|\aal_{\bs j}|,|\aal_{\bs j \bs \hat\imath}|}
\bb_{\XX}(\aal_{\bs j \bs \hat\imath})_{+\hat\imath+j} \\
&= \left({\textstyle \sum_{\hat\imath}}\, x_{|\aal_{\bs j}|,|\aal_{\bs j \bs \hat\imath}|}\right)
\bb_{\XX}(\aal_{\bs j})_{+j}.
\end{aligned}\]
From this equality \eqref{eq:expansion} simplifies:
\[
\begin{aligned}
\frac{\displaystyle \bb_{\XX}(\aal_{\bs j})_{+j} + {\textstyle \sum_{i \neq j}}\, x_{|\aal|,|\aal_{\bs i}|} \bb_{\XX}(\aal_{\bs i})_{+i}}{\displaystyle 1 + {\textstyle \sum_{i \neq j}}\, x_{|\aal|,|\aal_{\bs i}|}}
&=
\frac{\displaystyle (1 + {\textstyle \sum_{\hat\imath}}\, x_{|\aal_{\bs j}|,|\aal_{\bs j\bs \hat\imath}|})\bb_{\XX}(\aal_{\bs j})_{+j}}{\displaystyle 1 + {\textstyle \sum_{\hat\imath}}\, x_{|\aal_{\bs j}|,|\aal_{\bs j \bs \hat\imath}|}} \\
&=
\bb_{\XX}(\aal_{\bs j})_{+j}.
\end{aligned}
\]
\end{proof}

Properties (II) and (III) together mean that the nodes of \(R^d_{\XX,n}\) on an
edge are always mappings of the 1D node set \(X_n\). This is useful when
simplices appear in hybrid meshes with tensor-product cells, which often use
tensor products of 1D node sets, because common edges between the two cell
types will have the same nodes.

\hypertarget{interp}{%
\subsection{Interpolation properties}\label{interp}}

A problem discretized by the finite element method may require the
approximation of an arbitrary function \(f\) in \(\mathcal{P}_n(\Delta^d)\).
Certain problems have optimal projection operators for this purpose, such as
\(L^2\) projection or \(H^1\) projection, but these operators can only be
approximated with numerical integration rules, and may be implicit or
expensive. When Lagrange polynomials are used as a basis, interpolation at the
nodes is an appealing projection onto \(\mathcal{P}_n(\Delta^d)\), because it
requires the minimum number of function evaluations. Let
\(I_{X}:\mathcal{B}(\Delta^d) \to \mathcal{P}_n(\Delta^d)\) be the
interpolation operator defined by nodes \(X\) acting on bounded, measurable
functions on the \(d\)-simplex. The interpolation error can be bounded by
\[\|I_X f - f\|_{\infty} \leq (1 + \Lambda_n^{\max}(X)) \inf_{p \in
\mathcal{P}_n(\Delta^d)} \| p - f \|_{\infty},\] where \(\Lambda_n^{\max}(X)\) is
the Lebesgue constant, defined by the shape functions \(\Phi_X\) associated with
\(X\),
\[\Lambda_n^{\max}(X) = \max_{\xx \in \Delta^d} {\textstyle
\sum_{\varphi \in \Phi_X}}\ |\varphi(\xx)|.\]

\begin{figure}
\centering
\includegraphics{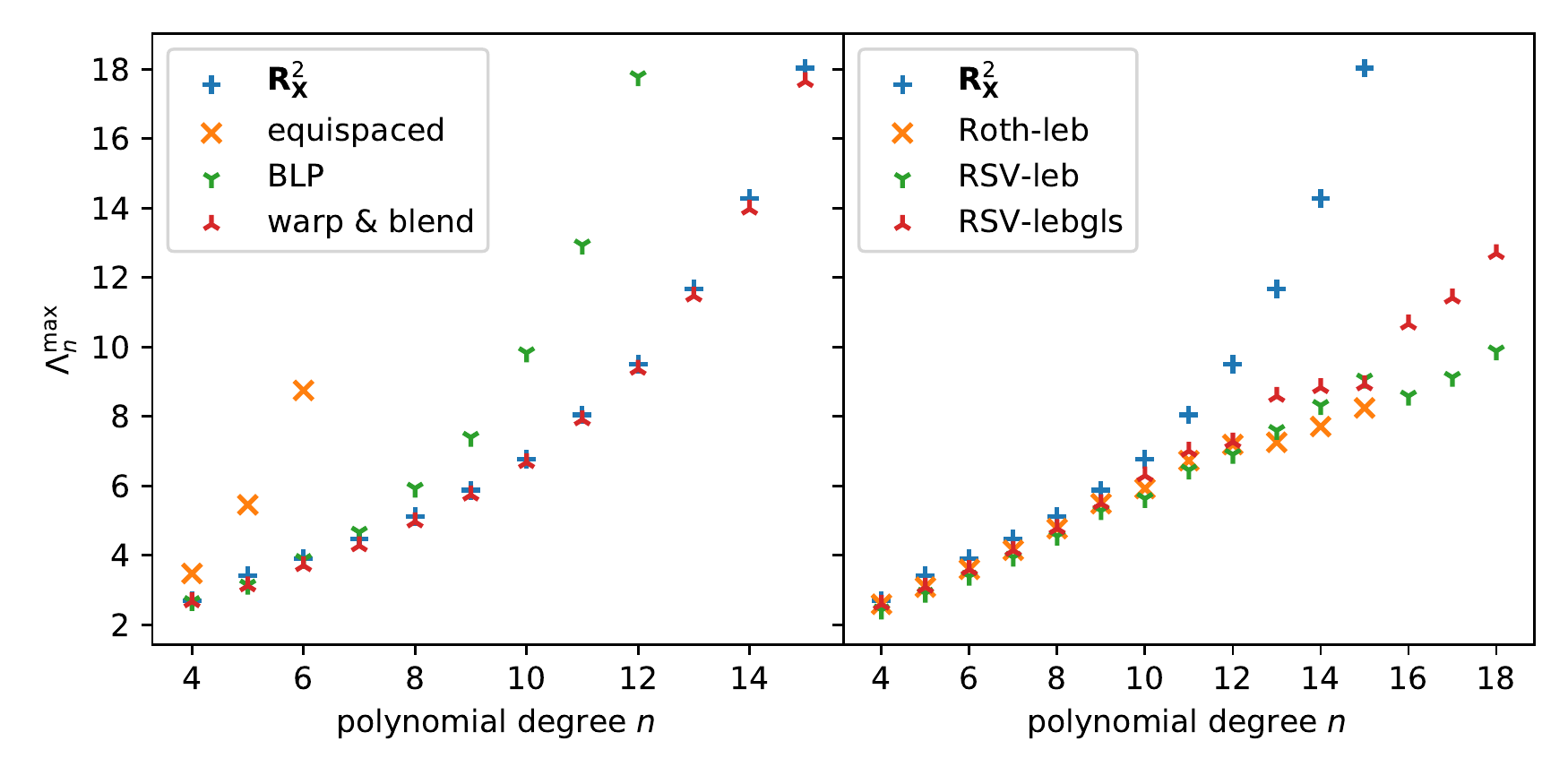}
\caption{\label{fig:leb}Lebesgue constants on the triangle, comparing \(\RR^2_{\XX}\) against node families defined explicitly (left) and implicitly (right).}
\end{figure}

\begin{figure}
\centering
\includegraphics{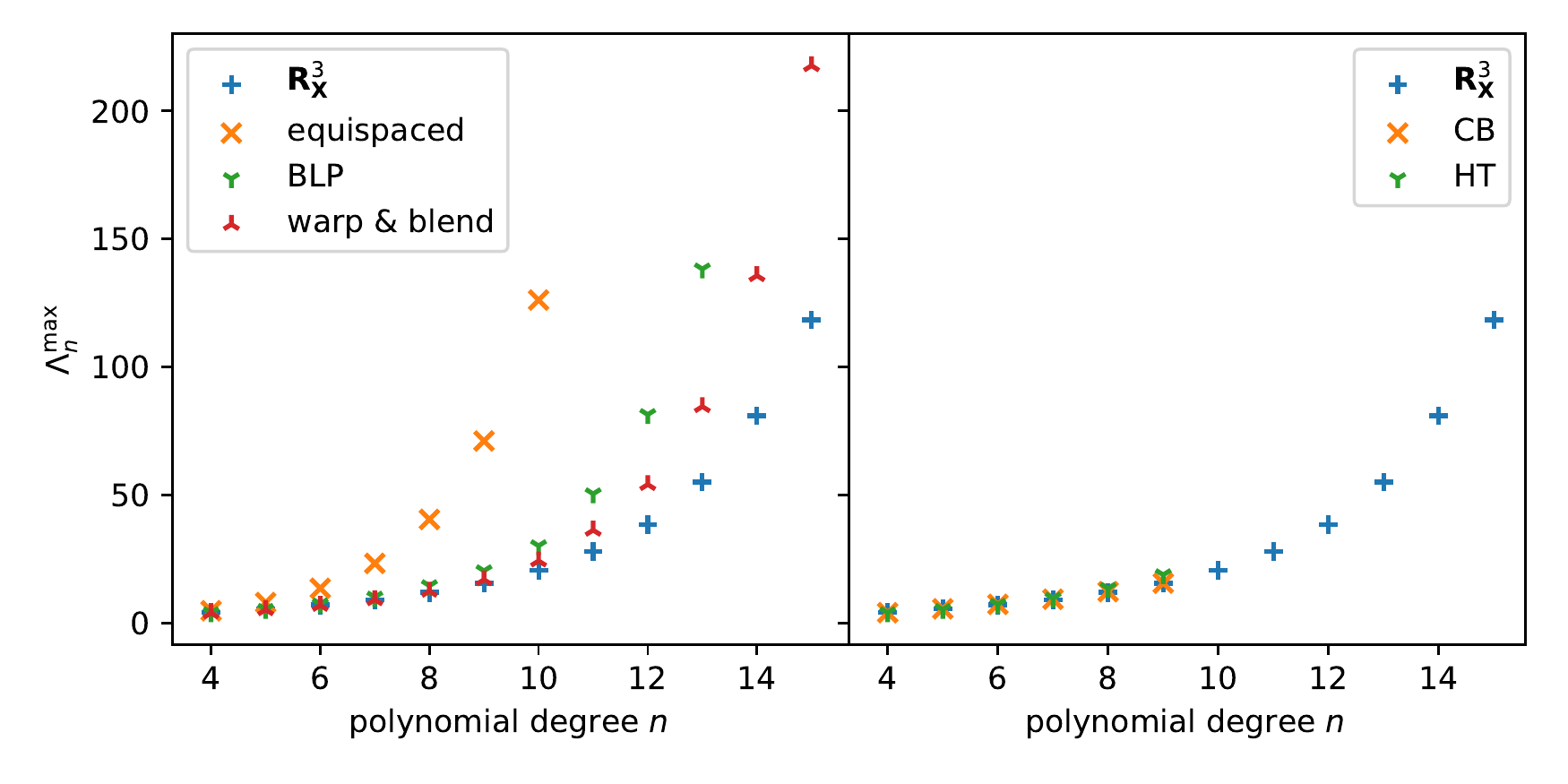}
\caption{\label{fig:lebthree}Lebesgue constants on the tetrahedron, comparing \(\RR^3_{\XX}\) against node families defined explicitly (left) and implicitly (right).}
\end{figure}

Lebesgue constants for \(\RR^d_{\XX}\)
are compared against some other node families on the triangle in Fig.~\ref{fig:leb} and on the tetrahedron in Fig.~\ref{fig:lebthree}.
These include:

\begin{itemize}
\tightlist
\item
  \textbf{equispaced:} Equispaced nodes, defined by
  \(b_{\text{eq},i}(\aal) = \alpha_i / |\aal|\).
\item
  \textbf{BLP:} The nodes of Blyth, Luo, \& Pozrikidis \autocite{BlPo06},\autocite{LuPo06}, which like the
  recursively defined nodes are based on the LGL nodes
  \(\XX_{\text{LGL}}\). If \(\aal > 0\),
  which indicates that the node will be in the interior of the simplex,
  they are defined by
  \[b_{\text{BLP},i}(\aal) =
  \frac{1}{|\aal|}(1 +
  |\aal|x_{|\aal|, \alpha_i} - {\textstyle
  \sum_j}\, x_{|\aal|, \alpha_j}).\]
  Points on the boundary are mapped from the same rule applied to the
  \((d-1)\)-simplex.
\item
  \textbf{warp \& blend:} The nodes of \textcite{Warb06}, which define the node
  location \(\bb_{\text{wb}}(\aal)\) as the image of
  the equispaced node \(\bb_{\text{eq}}(\aal)\) under a
  smooth bijection of the \(d\)-simplex. The bijection sends equispaced nodes
  to LGL nodes on the edges. The smooth map is nearly isoparametric, but a
  blending parameter is introduced that controls the distortion in the interior
  of the element, and optimal values of this blending parameter have been
  computed for \(n\) up to \(15\) in \(d=2\) and \(3\).
\item
  \textbf{Roth-leb:} Nodes for the triangle computed by \textcite{Roth05} by numerical
  minimization of \(\Lambda_n^{\max}\).
\item
  \textbf{RSV-leb:} Nodes for the triangle computed by \textcite{RaSV12} by numerical
  minimization of \(\Lambda_n^{\max}\).
\item
  \textbf{RSV-lebgls:} Nodes for the triangle computed by \textcite{RaSV12} by numerical
  minimization of \(\Lambda_n^{\max}\),
  subject to the constraints that the nodes remain symmetric and that the nodes
  on the edges be LGL nodes.
\item
  \textbf{CB:} Nodes for the tetrahedron computed by \textcite{ChBa96} by numerical
  minimization of the related interpolation metric \[\Lambda_n^2(X) =
  \int_{\Delta^d} \sum_{\varphi\in \Phi_X} \varphi(\xx)^2\ dx.\]
\item
  \textbf{HT:} Nodes for the tetrahedron computed by \textcite{HeTe00} as the equilibrium
  distribution of charged particles.
\end{itemize}

All of these node families except \emph{Roth-leb} and \emph{RSV-leb} are symmetric
for all \(n\), and all except \emph{equispaced}, \emph{Roth-leb}, and \emph{RSV-leb} have
edge traces that are LGL nodes. The \emph{equispaced}, \emph{BLP}, and \emph{warp \& blend}
nodes can be explicitly defined in any dimension and have the recursive
boundary property (III) from Section \ref{symmetry}.\footnote{The \emph{warp \& blend} nodes have this property
  if the same value of the blending parameter is used for each dimension.}

Both Figs.~\ref{fig:leb} and \ref{fig:lebthree} split the comparison of the
\(\RR^d_{\XX}\) node family into comparisons against
families that are explicitly defined and nodes that are implicitly defined as
the solution of an optimization problem.

\textbf{2D:} In 2D, the \(\RR^2_{\XX}\) node family has Lebesgue constants that are
not much worse than those for node families implicitly defined to minimize the
Lebesgue constant for \(n \leq 9\) (Fig.~\ref{fig:leb}, left). For \(n \geq 10\), the Lebesgue constant grows much faster for \(\RR^2_{\XX}\) than for the
best implicitly defined nodes.

Not coincidentally, at \(n=10\) the layout of implicitly defined nodes that
minimize the Lebesgue constant changes significantly. Until then, the
\emph{RSV-lebgls} nodes look ``lattice-like,'' as though they have been smoothly,
symmetrically, and monotonically mapped from the equispaced nodes, the same
as the explicit node families. At \(n=10\), however, this pattern changes
(Fig.~\ref{fig:lebgls}). This suggests that no node family that retains the
lattice-like structure, including \(\RR^d_{\XX}\), can attain a slow growth of
\(\Lambda_n^{\max}\) like the implicitly defined families.

\begin{figure}
\centering
\includegraphics{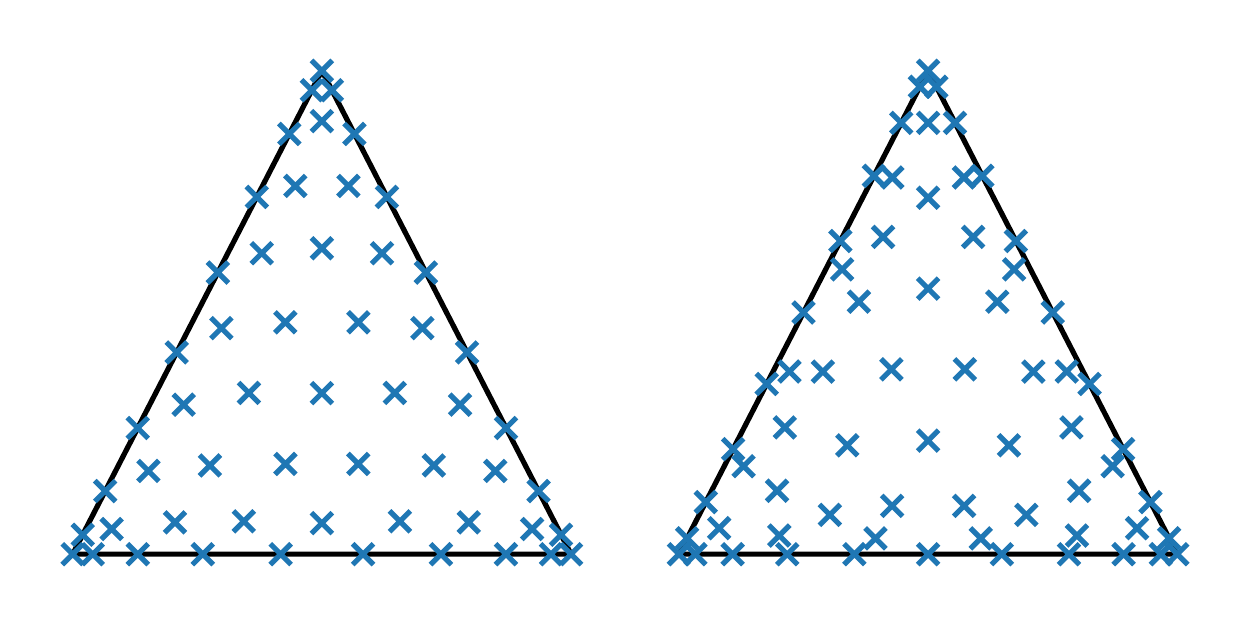}
\caption{\label{fig:lebgls}The LEBGLS nodes of \textcite{RaSV12}, showing the abrupt change in
layout between \(n=9\) (left) and \(n=10\) (right).}
\end{figure}

In comparison to the other explicitly defined nodes (Fig.~\ref{fig:leb},
right), the \(\RR^2_{\XX}\) family is nearly as good as the \emph{warp \& blend}
family, which has the best Lebesgue constants: \(\Lambda_n^{\max}\) is never more
than \(10\%\) different between them for \(n \leq 15\). In fact, despite the
differences in their definitions---\emph{warp \& blend} by continuous bijections,
\(\RR^2_{\XX}\) by recursion---the node families are remarkably similar for \(n \leq 15\): \(\|\bb_{\XX}(\alpha) - \bb_{\text{wb}}(\alpha)\| \leq 0.01\) for every
node in these node sets.

\textbf{3D:} In 3D there are no published examples of \(\Lambda_n^{\max}\)-optimal
node sets that have been numerically computed in the same way as in 2D.
\(\Lambda_n^{\max}(X)\) is a nonconvex function of the node coordinates in \(X\),
and the number of coordinates grows cubically with \(n\), so this is a
challenging optimization problem. Instead, the implicitly
defined node families \emph{CB} and \emph{HT} optimize simpler objectives:
the \(\Lambda_n^2\) interpolation metric and the electrostatic potential,
respectively, and these have only been computed to \(n \leq 9\).
There is little difference in \(\Lambda_n^{\max}\) between
\(\RR^{3}_{\XX}\) and these two families (Fig.~\ref{fig:lebthree}, left),
though it is slightly smaller than both for \(n \geq 6\).

In comparison to the explicitly defined node families \emph{BLP} and \emph{warp \& blend}
(Fig.~\ref{fig:lebthree}, right), there is little difference for \(n \leq 6\)
(all are within \(7\%\) of each other), but \(\RR^3_{\XX}\) is increasingly
superior for \(n \geq 7\). For \(n=15\), the largest for which the \emph{warp \& blend}
nodes' blending parameter has been optimized, the Lebesgue constant of
\(\RR^3_{\XX}\) is \(40\%\) smaller.

\hypertarget{asymptotic-interpolation-properties}{%
\subsection{Asymptotic interpolation properties}\label{asymptotic-interpolation-properties}}

A node family is good for approximation by interpolation if the
interpolants are known to converge for a large class of functions.
In particular, if \(f\) is analytic in the neighborhood of \(\Delta^d\),
then there is a sequence of polynomials \(p_n \rightrightarrows f\) (converging uniformly
on \(\Delta^d\)), so it is possible that given the right node family
\(\XX = \{X_n\}\) that \(I_{X_n} f \rightrightarrows f\) for all analytic \(f\) as well.

The weakest known sufficient condition that guarantees this for \(d > 1\)
is sub-exponential growth of the Lebesgue constant:
if \(\Lambda_n^{\max}(X_n)^{1/n} \to 1\), then \(I_{X_n} f \rightrightarrows f\) for \(f\) analytic in a neighborhood of \(\Delta^d\) \autocite{BBCL92}. The values of
\(\Lambda_n^{\max}(R^d_{\XX,n})\) that appeared in Figs.~\ref{fig:leb} \&
\ref{fig:lebthree} are tabulated in Table \ref{tab:lebtable}. In all tabulated values,
\((\Lambda_n^{\max})^{1/n}\) continues to increase instead of converging to 1, so they show no evidence
of sub-exponential growth.

\begin{longtable}[]{@{}rrrrr@{}}
\caption{\label{tab:lebtable} Lebesgue constants computed for \(\RR^d_{\XX}\)}\tabularnewline
\toprule
\begin{minipage}[b]{0.04\columnwidth}\raggedleft
\(n\)\strut
\end{minipage} & \begin{minipage}[b]{0.19\columnwidth}\raggedleft
\(\Lambda_n^{\max}(R^2_{\XX,n})\)\strut
\end{minipage} & \begin{minipage}[b]{0.22\columnwidth}\raggedleft
\(\Lambda_n^{\max}(R^2_{\XX,n})^{1/n}\)\strut
\end{minipage} & \begin{minipage}[b]{0.19\columnwidth}\raggedleft
\(\Lambda_n^{\max}(R^3_{\XX,n})\)\strut
\end{minipage} & \begin{minipage}[b]{0.22\columnwidth}\raggedleft
\(\Lambda_n^{\max}(R^3_{\XX,n})^{1/n}\)\strut
\end{minipage}\tabularnewline
\midrule
\endfirsthead
\toprule
\begin{minipage}[b]{0.04\columnwidth}\raggedleft
\(n\)\strut
\end{minipage} & \begin{minipage}[b]{0.19\columnwidth}\raggedleft
\(\Lambda_n^{\max}(R^2_{\XX,n})\)\strut
\end{minipage} & \begin{minipage}[b]{0.22\columnwidth}\raggedleft
\(\Lambda_n^{\max}(R^2_{\XX,n})^{1/n}\)\strut
\end{minipage} & \begin{minipage}[b]{0.19\columnwidth}\raggedleft
\(\Lambda_n^{\max}(R^3_{\XX,n})\)\strut
\end{minipage} & \begin{minipage}[b]{0.22\columnwidth}\raggedleft
\(\Lambda_n^{\max}(R^3_{\XX,n})^{1/n}\)\strut
\end{minipage}\tabularnewline
\midrule
\endhead
\begin{minipage}[t]{0.04\columnwidth}\raggedleft
4\strut
\end{minipage} & \begin{minipage}[t]{0.19\columnwidth}\raggedleft
2.67857\strut
\end{minipage} & \begin{minipage}[t]{0.22\columnwidth}\raggedleft
1.27931\strut
\end{minipage} & \begin{minipage}[t]{0.19\columnwidth}\raggedleft
4.09308\strut
\end{minipage} & \begin{minipage}[t]{0.22\columnwidth}\raggedleft
1.42237\strut
\end{minipage}\tabularnewline
\begin{minipage}[t]{0.04\columnwidth}\raggedleft
5\strut
\end{minipage} & \begin{minipage}[t]{0.19\columnwidth}\raggedleft
3.40745\strut
\end{minipage} & \begin{minipage}[t]{0.22\columnwidth}\raggedleft
1.27787\strut
\end{minipage} & \begin{minipage}[t]{0.19\columnwidth}\raggedleft
5.54727\strut
\end{minipage} & \begin{minipage}[t]{0.22\columnwidth}\raggedleft
1.40869\strut
\end{minipage}\tabularnewline
\begin{minipage}[t]{0.04\columnwidth}\raggedleft
6\strut
\end{minipage} & \begin{minipage}[t]{0.19\columnwidth}\raggedleft
3.90448\strut
\end{minipage} & \begin{minipage}[t]{0.22\columnwidth}\raggedleft
1.25486\strut
\end{minipage} & \begin{minipage}[t]{0.19\columnwidth}\raggedleft
7.16891\strut
\end{minipage} & \begin{minipage}[t]{0.22\columnwidth}\raggedleft
1.38859\strut
\end{minipage}\tabularnewline
\begin{minipage}[t]{0.04\columnwidth}\raggedleft
7\strut
\end{minipage} & \begin{minipage}[t]{0.19\columnwidth}\raggedleft
4.47897\strut
\end{minipage} & \begin{minipage}[t]{0.22\columnwidth}\raggedleft
1.23887\strut
\end{minipage} & \begin{minipage}[t]{0.19\columnwidth}\raggedleft
9.20205\strut
\end{minipage} & \begin{minipage}[t]{0.22\columnwidth}\raggedleft
1.37309\strut
\end{minipage}\tabularnewline
\begin{minipage}[t]{0.04\columnwidth}\raggedleft
8\strut
\end{minipage} & \begin{minipage}[t]{0.19\columnwidth}\raggedleft
5.10406\strut
\end{minipage} & \begin{minipage}[t]{0.22\columnwidth}\raggedleft
1.226\strut
\end{minipage} & \begin{minipage}[t]{0.19\columnwidth}\raggedleft
12.0671\strut
\end{minipage} & \begin{minipage}[t]{0.22\columnwidth}\raggedleft
1.36521\strut
\end{minipage}\tabularnewline
\begin{minipage}[t]{0.04\columnwidth}\raggedleft
9\strut
\end{minipage} & \begin{minipage}[t]{0.19\columnwidth}\raggedleft
5.87268\strut
\end{minipage} & \begin{minipage}[t]{0.22\columnwidth}\raggedleft
1.21738\strut
\end{minipage} & \begin{minipage}[t]{0.19\columnwidth}\raggedleft
15.5927\strut
\end{minipage} & \begin{minipage}[t]{0.22\columnwidth}\raggedleft
1.3569\strut
\end{minipage}\tabularnewline
\begin{minipage}[t]{0.04\columnwidth}\raggedleft
10\strut
\end{minipage} & \begin{minipage}[t]{0.19\columnwidth}\raggedleft
6.77248\strut
\end{minipage} & \begin{minipage}[t]{0.22\columnwidth}\raggedleft
1.21081\strut
\end{minipage} & \begin{minipage}[t]{0.19\columnwidth}\raggedleft
20.6234\strut
\end{minipage} & \begin{minipage}[t]{0.22\columnwidth}\raggedleft
1.35343\strut
\end{minipage}\tabularnewline
\begin{minipage}[t]{0.04\columnwidth}\raggedleft
11\strut
\end{minipage} & \begin{minipage}[t]{0.19\columnwidth}\raggedleft
8.04267\strut
\end{minipage} & \begin{minipage}[t]{0.22\columnwidth}\raggedleft
1.20867\strut
\end{minipage} & \begin{minipage}[t]{0.19\columnwidth}\raggedleft
28.034\strut
\end{minipage} & \begin{minipage}[t]{0.22\columnwidth}\raggedleft
1.35397\strut
\end{minipage}\tabularnewline
\begin{minipage}[t]{0.04\columnwidth}\raggedleft
12\strut
\end{minipage} & \begin{minipage}[t]{0.19\columnwidth}\raggedleft
9.49527\strut
\end{minipage} & \begin{minipage}[t]{0.22\columnwidth}\raggedleft
1.20631\strut
\end{minipage} & \begin{minipage}[t]{0.19\columnwidth}\raggedleft
38.6495\strut
\end{minipage} & \begin{minipage}[t]{0.22\columnwidth}\raggedleft
1.35601\strut
\end{minipage}\tabularnewline
\begin{minipage}[t]{0.04\columnwidth}\raggedleft
13\strut
\end{minipage} & \begin{minipage}[t]{0.19\columnwidth}\raggedleft
11.6647\strut
\end{minipage} & \begin{minipage}[t]{0.22\columnwidth}\raggedleft
1.208\strut
\end{minipage} & \begin{minipage}[t]{0.19\columnwidth}\raggedleft
55.1425\strut
\end{minipage} & \begin{minipage}[t]{0.22\columnwidth}\raggedleft
1.36132\strut
\end{minipage}\tabularnewline
\begin{minipage}[t]{0.04\columnwidth}\raggedleft
14\strut
\end{minipage} & \begin{minipage}[t]{0.19\columnwidth}\raggedleft
14.2678\strut
\end{minipage} & \begin{minipage}[t]{0.22\columnwidth}\raggedleft
1.20908\strut
\end{minipage} & \begin{minipage}[t]{0.19\columnwidth}\raggedleft
81.0374\strut
\end{minipage} & \begin{minipage}[t]{0.22\columnwidth}\raggedleft
1.36878\strut
\end{minipage}\tabularnewline
\begin{minipage}[t]{0.04\columnwidth}\raggedleft
15\strut
\end{minipage} & \begin{minipage}[t]{0.19\columnwidth}\raggedleft
18.0306\strut
\end{minipage} & \begin{minipage}[t]{0.22\columnwidth}\raggedleft
1.21265\strut
\end{minipage} & \begin{minipage}[t]{0.19\columnwidth}\raggedleft
118.42\strut
\end{minipage} & \begin{minipage}[t]{0.22\columnwidth}\raggedleft
1.37476\strut
\end{minipage}\tabularnewline
\bottomrule
\end{longtable}

In fact, \textcite{BBCL92} considered it an open question whether
explicitly computed node families with uniformly convergent interpolants
exist for \emph{any} nontrivial set in \(d > 1\). In the intervening time,
analogues of the Chebyshev polynomials have been found for domains
related to root systems \autocite{RyMu11}, but these domains are not simplices.
\textcite{BBCL12} considered the question still open for simplices twenty years later,
and it appears to still be open now.

\hypertarget{finite-element-matrix-conditioning}{%
\subsection{Finite element matrix conditioning}\label{finite-element-matrix-conditioning}}

Matrices that show up repeatedly in applications of the finite element method
include the mass matrix \(M_{ij} = \int_{\Delta^d} \varphi_{i} \varphi_{j}\ dx\)
and the stiffness matrix \(K_{ij} = \int_{\Delta^d} \nabla \varphi_{n,i} \cdot \nabla \varphi_{n,j}\ dx\). Let \(G_{ijk} = \partial_j \varphi_{k} (\xx_i)\)
(considered as a matrix in \(\mathbb{R}^{d|\Phi| \times |\Phi|}\)), and \(L_{ij} = \nabla \cdot \nabla \varphi_{j}(\xx_i)\). These are the nodal gradient and
Laplacian matrices, that appear in strong-form nodal discontinuous Galerkin
methods. The condition numbers of these matrices (using the definition
\(\kappa_2(A) = \|A\|_2\|A^{\dagger}\|_2\)) for \(\RR^d_{\XX}\) are compared
against the condition numbers for the equispaced, \emph{BLP}, \emph{warp \& blend}, and
\emph{RSV-lebgls} nodes in Fig.~\ref{fig:condition}. The condition number of \(M\)
is affine invariant and in a quasiuniform mesh bounds the condition number of a
fully assembled mass matrix \autocite{Wath87}, while the condition numbers of \(K\), \(G\),
and \(L\) depend on the choice of reference simplex: in this work, they are
computed with respect to the biunit simplex \(\Delta^d_{\text{bi}} = \{\xx \in \mathbb{R}^d: \xx \geq -1,\ \sum_i x_i \leq 2 - d\}\). The rankings of the node
families by these metrics are essentially the same as by the Lebesgue constant
in Section \ref{interp}.

\begin{figure}
\centering
\includegraphics{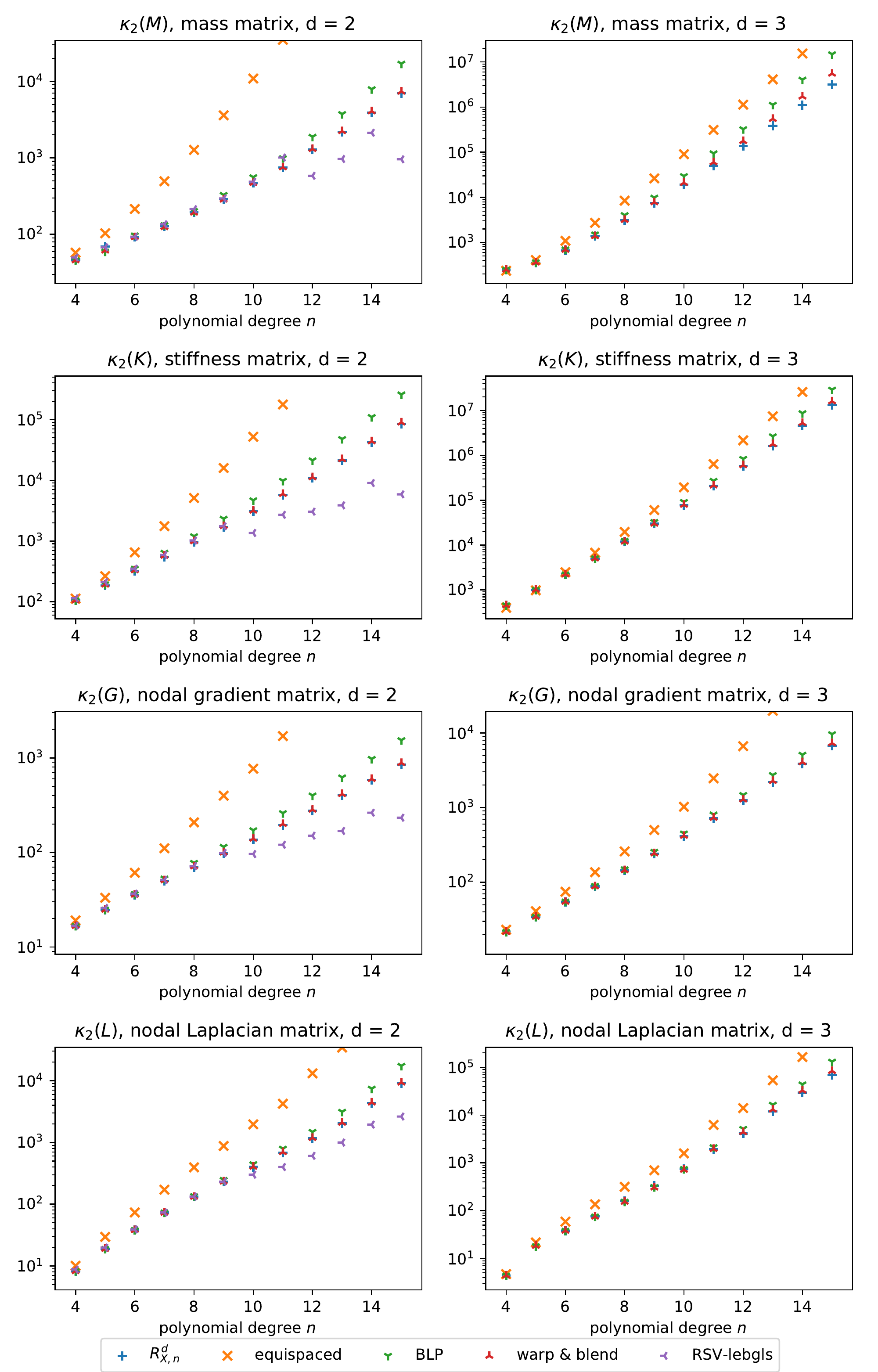}
\caption{\label{fig:condition}Condition numbers of finite element matrices.}
\end{figure}

\begin{longtable}[]{@{}rrrrrrrrr@{}}
\caption{\label{tab:condtable} Finite element matrix condition numbers \(\RR^2_{\XX}\)}\tabularnewline
\toprule
\begin{minipage}[b]{0.03\columnwidth}\raggedleft
\(n\)\strut
\end{minipage} & \begin{minipage}[b]{0.08\columnwidth}\raggedleft
\(\kappa_2(M)\)\strut
\end{minipage} & \begin{minipage}[b]{0.10\columnwidth}\raggedleft
\(\kappa_2(M)^{1/n}\)\strut
\end{minipage} & \begin{minipage}[b]{0.08\columnwidth}\raggedleft
\(\kappa_2(K)\)\strut
\end{minipage} & \begin{minipage}[b]{0.10\columnwidth}\raggedleft
\(\kappa_2(K)^{1/n}\)\strut
\end{minipage} & \begin{minipage}[b]{0.08\columnwidth}\raggedleft
\(\kappa_2(G)\)\strut
\end{minipage} & \begin{minipage}[b]{0.10\columnwidth}\raggedleft
\(\kappa_2(G)^{1/n}\)\strut
\end{minipage} & \begin{minipage}[b]{0.08\columnwidth}\raggedleft
\(\kappa_2(L)\)\strut
\end{minipage} & \begin{minipage}[b]{0.10\columnwidth}\raggedleft
\(\kappa_2(L)^{1/n}\)\strut
\end{minipage}\tabularnewline
\midrule
\endfirsthead
\toprule
\begin{minipage}[b]{0.03\columnwidth}\raggedleft
\(n\)\strut
\end{minipage} & \begin{minipage}[b]{0.08\columnwidth}\raggedleft
\(\kappa_2(M)\)\strut
\end{minipage} & \begin{minipage}[b]{0.10\columnwidth}\raggedleft
\(\kappa_2(M)^{1/n}\)\strut
\end{minipage} & \begin{minipage}[b]{0.08\columnwidth}\raggedleft
\(\kappa_2(K)\)\strut
\end{minipage} & \begin{minipage}[b]{0.10\columnwidth}\raggedleft
\(\kappa_2(K)^{1/n}\)\strut
\end{minipage} & \begin{minipage}[b]{0.08\columnwidth}\raggedleft
\(\kappa_2(G)\)\strut
\end{minipage} & \begin{minipage}[b]{0.10\columnwidth}\raggedleft
\(\kappa_2(G)^{1/n}\)\strut
\end{minipage} & \begin{minipage}[b]{0.08\columnwidth}\raggedleft
\(\kappa_2(L)\)\strut
\end{minipage} & \begin{minipage}[b]{0.10\columnwidth}\raggedleft
\(\kappa_2(L)^{1/n}\)\strut
\end{minipage}\tabularnewline
\midrule
\endhead
\begin{minipage}[t]{0.03\columnwidth}\raggedleft
4\strut
\end{minipage} & \begin{minipage}[t]{0.08\columnwidth}\raggedleft
4.7e+01\strut
\end{minipage} & \begin{minipage}[t]{0.10\columnwidth}\raggedleft
2.618\strut
\end{minipage} & \begin{minipage}[t]{0.08\columnwidth}\raggedleft
1.0e+02\strut
\end{minipage} & \begin{minipage}[t]{0.10\columnwidth}\raggedleft
3.196\strut
\end{minipage} & \begin{minipage}[t]{0.08\columnwidth}\raggedleft
1.7e+01\strut
\end{minipage} & \begin{minipage}[t]{0.10\columnwidth}\raggedleft
2.022\strut
\end{minipage} & \begin{minipage}[t]{0.08\columnwidth}\raggedleft
8.2e+00\strut
\end{minipage} & \begin{minipage}[t]{0.10\columnwidth}\raggedleft
1.691\strut
\end{minipage}\tabularnewline
\begin{minipage}[t]{0.03\columnwidth}\raggedleft
8\strut
\end{minipage} & \begin{minipage}[t]{0.08\columnwidth}\raggedleft
2.0e+02\strut
\end{minipage} & \begin{minipage}[t]{0.10\columnwidth}\raggedleft
1.933\strut
\end{minipage} & \begin{minipage}[t]{0.08\columnwidth}\raggedleft
9.5e+02\strut
\end{minipage} & \begin{minipage}[t]{0.10\columnwidth}\raggedleft
2.358\strut
\end{minipage} & \begin{minipage}[t]{0.08\columnwidth}\raggedleft
7.0e+01\strut
\end{minipage} & \begin{minipage}[t]{0.10\columnwidth}\raggedleft
1.700\strut
\end{minipage} & \begin{minipage}[t]{0.08\columnwidth}\raggedleft
1.3e+02\strut
\end{minipage} & \begin{minipage}[t]{0.10\columnwidth}\raggedleft
1.840\strut
\end{minipage}\tabularnewline
\begin{minipage}[t]{0.03\columnwidth}\raggedleft
16\strut
\end{minipage} & \begin{minipage}[t]{0.08\columnwidth}\raggedleft
1.3e+04\strut
\end{minipage} & \begin{minipage}[t]{0.10\columnwidth}\raggedleft
1.808\strut
\end{minipage} & \begin{minipage}[t]{0.08\columnwidth}\raggedleft
1.7e+05\strut
\end{minipage} & \begin{minipage}[t]{0.10\columnwidth}\raggedleft
2.124\strut
\end{minipage} & \begin{minipage}[t]{0.08\columnwidth}\raggedleft
1.2e+03\strut
\end{minipage} & \begin{minipage}[t]{0.10\columnwidth}\raggedleft
1.561\strut
\end{minipage} & \begin{minipage}[t]{0.08\columnwidth}\raggedleft
1.9e+04\strut
\end{minipage} & \begin{minipage}[t]{0.10\columnwidth}\raggedleft
1.848\strut
\end{minipage}\tabularnewline
\begin{minipage}[t]{0.03\columnwidth}\raggedleft
24\strut
\end{minipage} & \begin{minipage}[t]{0.08\columnwidth}\raggedleft
2.8e+06\strut
\end{minipage} & \begin{minipage}[t]{0.10\columnwidth}\raggedleft
1.856\strut
\end{minipage} & \begin{minipage}[t]{0.08\columnwidth}\raggedleft
6.3e+07\strut
\end{minipage} & \begin{minipage}[t]{0.10\columnwidth}\raggedleft
2.113\strut
\end{minipage} & \begin{minipage}[t]{0.08\columnwidth}\raggedleft
2.8e+04\strut
\end{minipage} & \begin{minipage}[t]{0.10\columnwidth}\raggedleft
1.532\strut
\end{minipage} & \begin{minipage}[t]{0.08\columnwidth}\raggedleft
7.4e+06\strut
\end{minipage} & \begin{minipage}[t]{0.10\columnwidth}\raggedleft
1.933\strut
\end{minipage}\tabularnewline
\begin{minipage}[t]{0.03\columnwidth}\raggedleft
32\strut
\end{minipage} & \begin{minipage}[t]{0.08\columnwidth}\raggedleft
8.0e+08\strut
\end{minipage} & \begin{minipage}[t]{0.10\columnwidth}\raggedleft
1.898\strut
\end{minipage} & \begin{minipage}[t]{0.08\columnwidth}\raggedleft
2.5e+10\strut
\end{minipage} & \begin{minipage}[t]{0.10\columnwidth}\raggedleft
2.114\strut
\end{minipage} & \begin{minipage}[t]{0.08\columnwidth}\raggedleft
6.2e+05\strut
\end{minipage} & \begin{minipage}[t]{0.10\columnwidth}\raggedleft
1.517\strut
\end{minipage} & \begin{minipage}[t]{0.08\columnwidth}\raggedleft
3.2e+09\strut
\end{minipage} & \begin{minipage}[t]{0.10\columnwidth}\raggedleft
1.982\strut
\end{minipage}\tabularnewline
\bottomrule
\end{longtable}

\begin{longtable}[]{@{}rrrrrrrrr@{}}
\caption{\label{tab:condtable3} Finite element matrix condition numbers \(\RR^3_{\XX}\)}\tabularnewline
\toprule
\begin{minipage}[b]{0.03\columnwidth}\raggedleft
\(n\)\strut
\end{minipage} & \begin{minipage}[b]{0.08\columnwidth}\raggedleft
\(\kappa_2(M)\)\strut
\end{minipage} & \begin{minipage}[b]{0.10\columnwidth}\raggedleft
\(\kappa_2(M)^{1/n}\)\strut
\end{minipage} & \begin{minipage}[b]{0.08\columnwidth}\raggedleft
\(\kappa_2(K)\)\strut
\end{minipage} & \begin{minipage}[b]{0.10\columnwidth}\raggedleft
\(\kappa_2(K)^{1/n}\)\strut
\end{minipage} & \begin{minipage}[b]{0.08\columnwidth}\raggedleft
\(\kappa_2(G)\)\strut
\end{minipage} & \begin{minipage}[b]{0.10\columnwidth}\raggedleft
\(\kappa_2(G)^{1/n}\)\strut
\end{minipage} & \begin{minipage}[b]{0.08\columnwidth}\raggedleft
\(\kappa_2(L)\)\strut
\end{minipage} & \begin{minipage}[b]{0.10\columnwidth}\raggedleft
\(\kappa_2(L)^{1/n}\)\strut
\end{minipage}\tabularnewline
\midrule
\endfirsthead
\toprule
\begin{minipage}[b]{0.03\columnwidth}\raggedleft
\(n\)\strut
\end{minipage} & \begin{minipage}[b]{0.08\columnwidth}\raggedleft
\(\kappa_2(M)\)\strut
\end{minipage} & \begin{minipage}[b]{0.10\columnwidth}\raggedleft
\(\kappa_2(M)^{1/n}\)\strut
\end{minipage} & \begin{minipage}[b]{0.08\columnwidth}\raggedleft
\(\kappa_2(K)\)\strut
\end{minipage} & \begin{minipage}[b]{0.10\columnwidth}\raggedleft
\(\kappa_2(K)^{1/n}\)\strut
\end{minipage} & \begin{minipage}[b]{0.08\columnwidth}\raggedleft
\(\kappa_2(G)\)\strut
\end{minipage} & \begin{minipage}[b]{0.10\columnwidth}\raggedleft
\(\kappa_2(G)^{1/n}\)\strut
\end{minipage} & \begin{minipage}[b]{0.08\columnwidth}\raggedleft
\(\kappa_2(L)\)\strut
\end{minipage} & \begin{minipage}[b]{0.10\columnwidth}\raggedleft
\(\kappa_2(L)^{1/n}\)\strut
\end{minipage}\tabularnewline
\midrule
\endhead
\begin{minipage}[t]{0.03\columnwidth}\raggedleft
4\strut
\end{minipage} & \begin{minipage}[t]{0.08\columnwidth}\raggedleft
2.5e+02\strut
\end{minipage} & \begin{minipage}[t]{0.10\columnwidth}\raggedleft
3.977\strut
\end{minipage} & \begin{minipage}[t]{0.08\columnwidth}\raggedleft
4.5e+02\strut
\end{minipage} & \begin{minipage}[t]{0.10\columnwidth}\raggedleft
4.615\strut
\end{minipage} & \begin{minipage}[t]{0.08\columnwidth}\raggedleft
2.2e+01\strut
\end{minipage} & \begin{minipage}[t]{0.10\columnwidth}\raggedleft
2.158\strut
\end{minipage} & \begin{minipage}[t]{0.08\columnwidth}\raggedleft
4.4e+00\strut
\end{minipage} & \begin{minipage}[t]{0.10\columnwidth}\raggedleft
1.449\strut
\end{minipage}\tabularnewline
\begin{minipage}[t]{0.03\columnwidth}\raggedleft
8\strut
\end{minipage} & \begin{minipage}[t]{0.08\columnwidth}\raggedleft
3.1e+03\strut
\end{minipage} & \begin{minipage}[t]{0.10\columnwidth}\raggedleft
2.734\strut
\end{minipage} & \begin{minipage}[t]{0.08\columnwidth}\raggedleft
1.2e+04\strut
\end{minipage} & \begin{minipage}[t]{0.10\columnwidth}\raggedleft
3.231\strut
\end{minipage} & \begin{minipage}[t]{0.08\columnwidth}\raggedleft
1.4e+02\strut
\end{minipage} & \begin{minipage}[t]{0.10\columnwidth}\raggedleft
1.862\strut
\end{minipage} & \begin{minipage}[t]{0.08\columnwidth}\raggedleft
1.6e+02\strut
\end{minipage} & \begin{minipage}[t]{0.10\columnwidth}\raggedleft
1.889\strut
\end{minipage}\tabularnewline
\begin{minipage}[t]{0.03\columnwidth}\raggedleft
12\strut
\end{minipage} & \begin{minipage}[t]{0.08\columnwidth}\raggedleft
1.4e+05\strut
\end{minipage} & \begin{minipage}[t]{0.10\columnwidth}\raggedleft
2.682\strut
\end{minipage} & \begin{minipage}[t]{0.08\columnwidth}\raggedleft
5.8e+05\strut
\end{minipage} & \begin{minipage}[t]{0.10\columnwidth}\raggedleft
3.022\strut
\end{minipage} & \begin{minipage}[t]{0.08\columnwidth}\raggedleft
1.3e+03\strut
\end{minipage} & \begin{minipage}[t]{0.10\columnwidth}\raggedleft
1.812\strut
\end{minipage} & \begin{minipage}[t]{0.08\columnwidth}\raggedleft
4.1e+03\strut
\end{minipage} & \begin{minipage}[t]{0.10\columnwidth}\raggedleft
2.001\strut
\end{minipage}\tabularnewline
\begin{minipage}[t]{0.03\columnwidth}\raggedleft
16\strut
\end{minipage} & \begin{minipage}[t]{0.08\columnwidth}\raggedleft
9.3e+06\strut
\end{minipage} & \begin{minipage}[t]{0.10\columnwidth}\raggedleft
2.726\strut
\end{minipage} & \begin{minipage}[t]{0.08\columnwidth}\raggedleft
3.8e+07\strut
\end{minipage} & \begin{minipage}[t]{0.10\columnwidth}\raggedleft
2.979\strut
\end{minipage} & \begin{minipage}[t]{0.08\columnwidth}\raggedleft
1.2e+04\strut
\end{minipage} & \begin{minipage}[t]{0.10\columnwidth}\raggedleft
1.798\strut
\end{minipage} & \begin{minipage}[t]{0.08\columnwidth}\raggedleft
1.8e+05\strut
\end{minipage} & \begin{minipage}[t]{0.10\columnwidth}\raggedleft
2.132\strut
\end{minipage}\tabularnewline
\bottomrule
\end{longtable}

In Tables \ref{tab:condtable} and \ref{tab:condtable3} the growth rates of
these condition numbers can be assessed. The values of \(\kappa_2(M)^{1/n}\),
\(\kappa_2(K)^{1/n}\) and \(\kappa_2(L)^{1/n}\) are not monotonically decreasing in
both dimensions for values of \(n\) that have been calculated, which suggests
super-exponential growth. \(\kappa_2(G)^{1/n}\) appears to be monotonically
decreasing towards some limit \(\gamma_d > 1\) for \(d=2\), and \(d=3\), which
suggests exponential growth, but there is no proof of this fact.

\hypertarget{finite-element-matrix-efficiency}{%
\subsection{Finite element matrix efficiency}\label{finite-element-matrix-efficiency}}

To evaluate the basis functions of \(R^d_{\XX,n}\) at a set of nodes \(Q\),
one can compute the Vandermonde matrices \(V_{R^d_{\XX,n}}\) and \(V_Q\)
with respect to a stable basis for \(\mathcal{R}_n(\Delta^d)\), such
as the Proriol-Koornwinder-Dubiner basis \autocite{Pror57},\autocite{Koor75},\autocite{Dubi91},
and assemble \(V_Q V_{R^d_{\XX,n}}^{-1}\). This approach goes back at least to \autocite{WaPH00}, and was improved with a singularity-free evaluation of the basis by \textcite{Kirb10b}. Assuming \(|Q|\in \Theta(n^d)\),
the cost of constructing this matrix is \(\Theta(n^{3d})\). There is no structure in \(R^d_{\XX,n}\)
that would allow for fast application to a vector of nodal coefficients,
so the cost of a matrix vector product is \(\Theta(n^{2d})\). The same costs
hold for each directional derivative of the basis functions.

There appears to be no Lagrange polynomial basis for \(\mathcal{R}_n(\Delta^d)\)
that improves on this for \(d > 1\), so all of the node families discussed above
are equal with respect to this metric. It must be noted, however, that outside
of Lagrange bases are bases that have fast algorithms, either through
hierarchical construction, like Bernstein-Bézier polynomials, or through
generalized tensor-product constructions related to the Duffy transformation,
like the basis of \textcite{ShKa95}. The Bernstein-Bézier basis has been the subject of
more recent work, and has fast algorithms that allow for optimal construction
in \(\Theta(n^{2d})\) and application in \(\Theta(n^d)\) for these matrices, for
constant coefficients matrices without quadrature \autocite{Kirb10}, for evaluation at
the Stroud quadrature points \autocite{AiAD11}, and for the inverse of the mass matrix
\autocite{Kirb16}. A drawback of the Bernstein-Bézier basis is mass-matrix condition
numbers \(\kappa_2(n,d) = \binom{2n + d}{n}\) that are worse even than equispaced
nodes, though recent work by \textcite{AlKi20} shows that a condition number that uses a
matrix norm based on the \(L_2\) norm of the reconstructed polynomials grows like
\(\sqrt{\kappa_2(n,d)}\).

\hypertarget{ease-of-computation-and-implementation}{%
\subsection{Ease of computation and implementation}\label{ease-of-computation-and-implementation}}

Implicitly defined node families, including \emph{Roth-leb}, \emph{RSV-leb}, \emph{RSV-lebgls},
\emph{CB}, and \emph{HT} from Section~\ref{interp} and others not discussed, require
the solution of an optimization problem over the choice of node coordinates, a
problem size that, even with symmetries enforced, is \(\Theta(n^d)\) in \(n\).
Objective functions like \(\Lambda_n^{\max}(X)\) are quite nonconvex, so care must
be taken to avoid local minima. It is fair to characterize these node sets as
relatively expensive to compute from scratch.

The ease of implementing node sets from a node family is distinct from the
computational complexity of computing the node sets from scratch. Most of the
implicitly defined node sets discussed in this work have published node sets
for moderate values of \(n\) for \(d=2\) \autocite{ChBa95}\autocite{Hest98}\autocite{RaSV12}, and a few for
\(d=3\) \autocite{ChBa96}\autocite{HeTe00}.

Of the explicitly defined node families discussed in this work, the
\emph{equispaced} and \emph{BLP} nodes are the cheapest to compute: the former requires
\(\Theta(d)\) operations and \(\Theta(1)\) workspace, the latter requires
\(\Theta(d^2)\) operations and \(\Theta(1)\) workspace per node. The \emph{warp \&
blend} nodes additionally require \(d+1\) evaluations of 1D Jacobi polynomials up
to degree \(n\) at each node (one per facet of the simplex) and one \(\Theta(n^3)\)
inversion of a 1D Vandermonde matrix of size \(n+1\) per node set.

The computational complexity of computing one node \(\bb_{\XX}(\aal)\) in
isolation by the rule \eqref{eq:recursive} satisfies the recursion \(T(d) = (d+1)T(d-1) + \Theta(d^2)\), which implies \(T(d)\in \Theta((d+1)!)\). The workspace
satisfies the recursion \(S(d) = S(d-1) + \Theta(d)\), so \(S(d)\in \Theta(d^2)\).
Neither of these are a concern for \(d=2\) or \(3\).

If the nodes must be computed for higher dimensions, the cost of computing a
full node set can be reduced by caching the lower-dimensional nodes. Then the
cost of computing the node sets \(\{R^d_{\XX,i}\}_{i=0}^n\) with caching
satisfies the recursion \(T(d,n) = T(d-1,n) + \Theta(\binom{n+d+1}{d+1}d^2)\),
which implies \(T(d,n) \in O(\binom{n+d+2}{d+1}d^2)\), leading to an amortized
cost per node that is \(O((n+d)d^2/n)\). The workspace with caching satisfies the
recursion \(S(d,n) = S(d-1,n) + \Theta(\binom{n+d+1}{d+1}d)\), so \(S(d,n) \in O(\binom{n+d+2}{d+1}d)\), which is an amortized space per node that is \(O((n+d)d/n)\).

In terms of implementation, once the 1D node family \(\XX\) is available,
the code to compute \(\bb_{\XX}(\aal)\) is very short. Here is an example
implementation in python:

\begin{Shaded}
\begin{Highlighting}[]
\ImportTok{import}\NormalTok{ numpy }\ImportTok{as}\NormalTok{ np}

\KeywordTok{def}\NormalTok{ recursive(alpha, family):}
    \CommentTok{'''The barycentric d-simplex coordinates for a multi-index}
\CommentTok{    alpha with length d+1 and sum n, based on a 1D node family.'''}
\NormalTok{    d }\OperatorTok{=} \BuiltInTok{len}\NormalTok{(alpha) }\OperatorTok{-} \DecValTok{1}
\NormalTok{    n }\OperatorTok{=} \BuiltInTok{sum}\NormalTok{(alpha)}
\NormalTok{    xn }\OperatorTok{=}\NormalTok{ family[n]}
\NormalTok{    b }\OperatorTok{=}\NormalTok{ np.zeros((d}\OperatorTok{+}\DecValTok{1}\NormalTok{,))}
    \ControlFlowTok{if}\NormalTok{ d }\OperatorTok{==} \DecValTok{1}\NormalTok{:}
\NormalTok{        b[:] }\OperatorTok{=}\NormalTok{ xn[[alpha[}\DecValTok{0}\NormalTok{], alpha[}\DecValTok{1}\NormalTok{]]]}
        \ControlFlowTok{return}\NormalTok{ b}
\NormalTok{    weight }\OperatorTok{=} \FloatTok{0.}
    \ControlFlowTok{for}\NormalTok{ i }\KeywordTok{in} \BuiltInTok{range}\NormalTok{(d}\OperatorTok{+}\DecValTok{1}\NormalTok{):}
\NormalTok{        alpha_noti }\OperatorTok{=}\NormalTok{ alpha[:i] }\OperatorTok{+}\NormalTok{ alpha[i}\OperatorTok{+}\DecValTok{1}\NormalTok{:]}
\NormalTok{        w }\OperatorTok{=}\NormalTok{ xn[n }\OperatorTok{-}\NormalTok{ alpha[i]]}
\NormalTok{        br }\OperatorTok{=}\NormalTok{ recursive(alpha_noti, family)}
\NormalTok{        b[:i] }\OperatorTok{+=}\NormalTok{ w }\OperatorTok{*}\NormalTok{ br[:i]}
\NormalTok{        b[i}\OperatorTok{+}\DecValTok{1}\NormalTok{:] }\OperatorTok{+=}\NormalTok{ w }\OperatorTok{*}\NormalTok{ br[i:]}
\NormalTok{        weight }\OperatorTok{+=}\NormalTok{ w}
\NormalTok{    b }\OperatorTok{/=}\NormalTok{ weight}
    \ControlFlowTok{return}\NormalTok{ b}
\end{Highlighting}
\end{Shaded}

A reference python implementation, which includes all numerical methods used to
evaluate and compare against other node families in this work, is available as
the \href{https://tisaac.gitlab.io/recursivenodes}{\texttt{recursivenodes}} package \autocite{Isaa20}. The website for the package \autocite{Isaac20b}
hosts a version of this manuscript showing how it was used to
generate all figures and tables.

\hypertarget{using-1d-families-other-than-lgl-nodes}{%
\section{Using 1D families other than LGL nodes}\label{using-1d-families-other-than-lgl-nodes}}

The analysis and comparison in Section~\ref{comparison} was conducted
under the assumption that the 1D node family \(\XX\) was \(\XX_{\text{LGL}}\),
the Lobatto-Gauss-Legendre nodes on the interval \([0,1]\). The recursive
rule \eqref{eq:recursive} allows for an arbitrary 1D node family. While
\(\XX_{\text{LGL}}\) appears to be the best choice according to the metrics
in Section~\ref{comparison}, for completeness a few alternate choices of \(\XX\) are presented here.

\textbf{\(\XX_{\text{eq}}\) (equispaced):} It is not surprising, given the discussion
in Section~\ref{intuition} where equispaced nodes provided the intuition
behind the recursive rule, that using \(\XX = \XX_{\text{eq}}\) reproduces the
equispaced nodes, \(\bb_{\XX_{\text{eq}}} (\aal) = \bb_{\text{eq}}(\aal)\).

\textbf{\(\XX_{\text{LGC}}\) (Lobatto-Gauss-Chebyshev):} The 1D Lobatto-Gauss-Chebyshev (LGC) node family has
good interpolation properties in 1D while being a nested family, with
\(X_{n,\text{LGC}} \subset X_{2n,\text{LGC}}.\) The recursive nodes
\(\RR^d_{\XX_{\text{LGC}}}\) inherit the nested property, as demonstrated in
Fig.~\ref{fig:lgc}, left.

\textbf{\(\XX_{\text{GL}}\) (Gauss-Legendre):} The 1D Gauss-Legendre (GC) node family
does not include the endpoints \(0\) and \(1\), but can still be used to construct
nodes. Property (III) in Section~\ref{symmetry} does not hold, and in fact
all nodes will be in the interior of the \(d\)-simplex as demonstrated in
Fig.~\ref{fig:lgc}, right.

\begin{figure}
\centering
\includegraphics{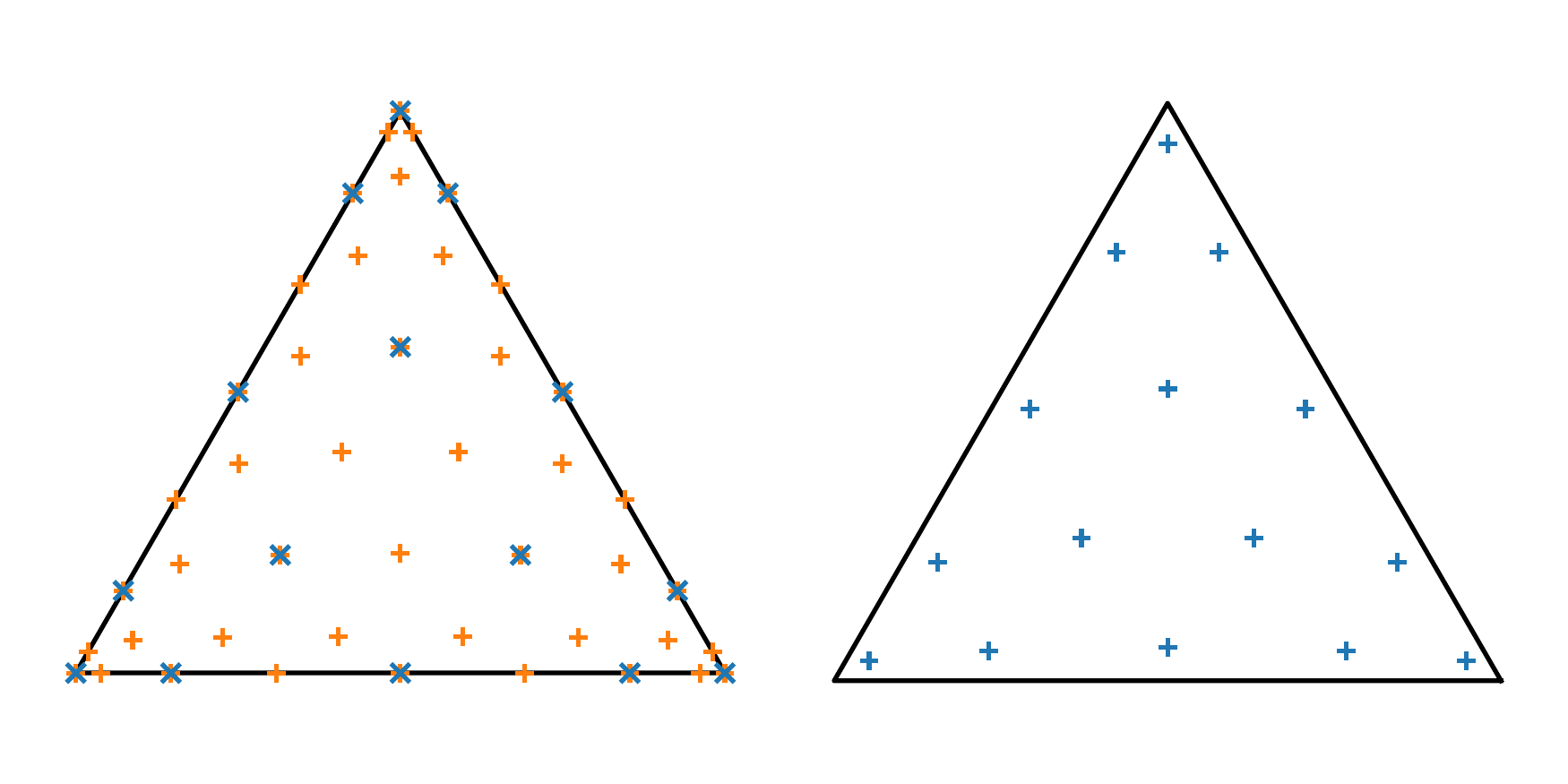}
\caption{\label{fig:lgc}(left) The node sets \(R^2_{\XX_{\text{LGC}},4}\) and
\(R^2_{\XX_{\text{LGC}},8}\), demonstrating nested node sets, a property
inherited from the 1D LGC node family. (right) The node set
\(R^2_{\XX_{\text{GL}},4}\), demonstrating node sets contained in the interior.}
\end{figure}

The metrics from Section~\ref{comparison}
are used to compare \(\RR^d_{\XX}\) for \(\XX_{\text{LGL}}\),
\(\XX_{\text{LGC}}\), and \(\XX_{\text{GL}}\) in
Fig.~\ref{fig:altcondition}.\footnote{The only metric omitted is the condition number
  of the nodal Laplacian matrix, where the node families show little distinction.}
The results for \(d=2\) and \(d=3\) resemble the results of the 1D node families,
with GL nodes having worse interpolation properties but better conditioned mass matrices
than either set of Lobatto nodes, and with LGC nodes having interpolation
properties and matrix condition numbers that similar but slightly worse to LGL
nodes. The most interesting trend to be observed in Fig.~\ref{fig:altcondition}
is that, while the growth rate of the Lebesgue constant for the GL nodes is worse than
the Lobatto nodes in 2D, it is much closer in 3D.

\begin{figure}
\centering
\includegraphics{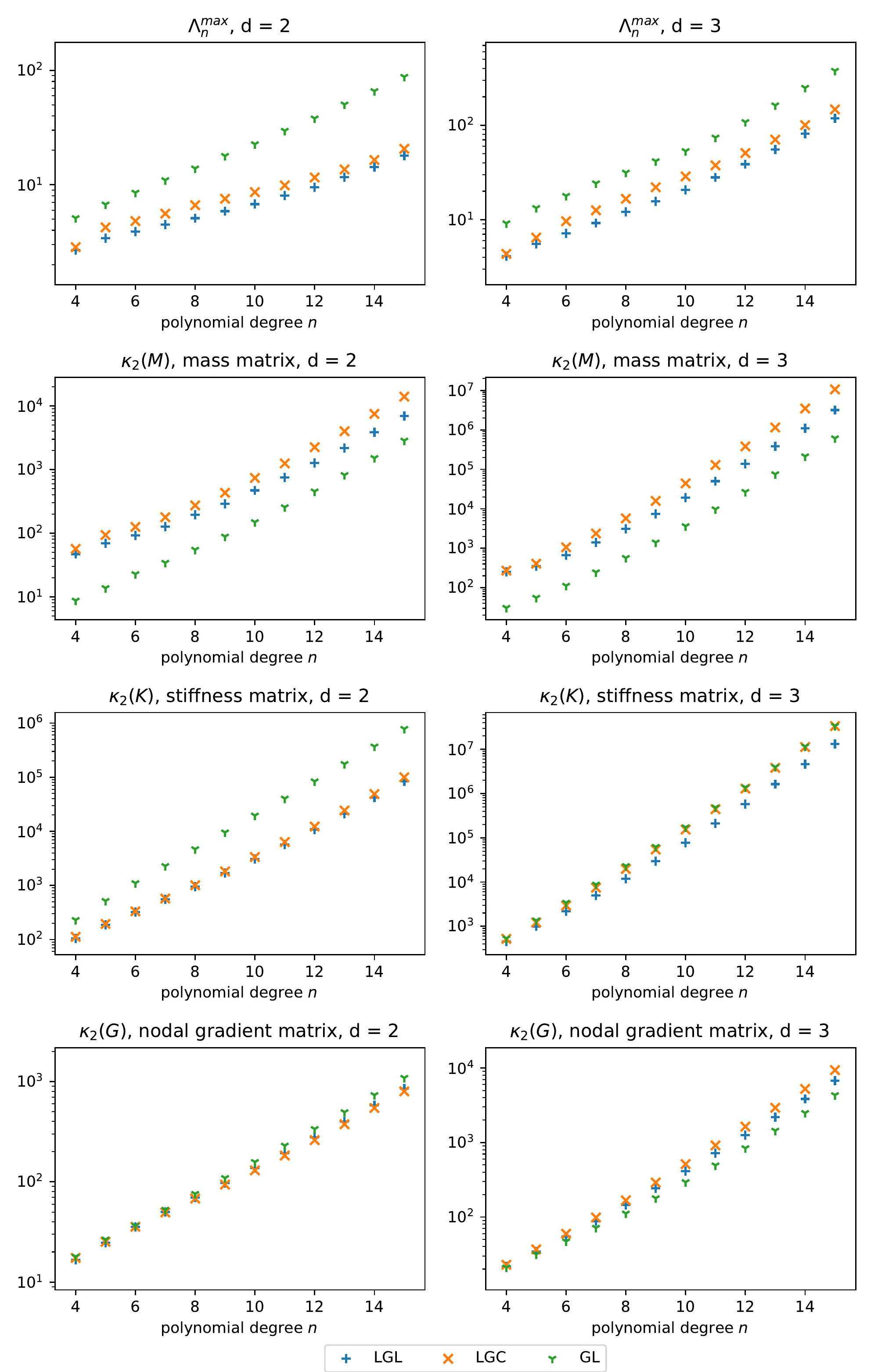}
\caption{\label{fig:altcondition}Comparing \(\RR^d_{\XX_{\text{LGL}}}\), \(\RR^d_{\XX_{\text{LGC}}}\), and \(\RR^d_{\XX_\text{GL}}\) according to metrics from Section~\ref{comparison}.}
\end{figure}

\hypertarget{interpolation-tests}{%
\section{Interpolation tests}\label{interpolation-tests}}

To show that the bounds implied by the Lebesgue constants in \ref{interp} are
in line with interpolation errors in practice, this section compares those
errors for the recursively constructed nodes against other node families for
two benchmark functions that have appeared previously.

The first function is\\
\begin{equation}
\label{eq:fA}
f_A(\xx) = \prod_{i} (x_i + 1) \cosh \left(\sum_i x_i - 1\right),
\end{equation}\\
which has appeared in \autocite{Hein05}\autocite{Warb06}\autocite{ChWa15} as an example of a smooth,
non-polynomial function to which even the equispaced polynomial interpolants
converge. In Table \ref{tab:interperr_smooth}, the absolute interpolation
errors \(\|I_X (f_A) - f_A\|_\infty\) on the biunit simplex is tested for node
families that appeared in Section \ref{comparison}.

The next is the ``Witch of Agnesi'' function,\\
\begin{equation}
\label{eq:fB}
f_B(\xx) = \frac{1}{1 + \alpha |\xx|^2},
\end{equation}\\
which for \(\alpha = 25\) is the classic Runge function, for which the equispaced
interpolants diverge in some domains. Table \ref{tab:interperr_runge} reports
\(\|I_X (f_B) - f_B\|_\infty\) on an equilateral simplex centered at
the origin with edge length 2. For \(d=2\), the standard \(\alpha = 25\) is used, but
for \(d=3\), \(\alpha = 60\) is used instead to make errors of the equispaced interpolants
roughly the same as for \(d=2\).

\begin{longtable}[]{@{}rrrrrrr@{}}
\caption{Interpolation errors \(\|I_X(f_A) - f_A\|_\infty\)\label{tab:interperr_smooth}}\tabularnewline
\toprule
\(d\) & \(n\) & equispaced & BLP & warp \& blend & RSV-lebgls & \({\RR}^d_{\XX}\)\tabularnewline
\midrule
\endfirsthead
\toprule
\(d\) & \(n\) & equispaced & BLP & warp \& blend & RSV-lebgls & \({\RR}^d_{\XX}\)\tabularnewline
\midrule
\endhead
2 & 6 & 3.6e-04 & 2.6e-04 & 2.4e-04 & 2.4e-04 & 2.2e-04\tabularnewline
2 & 9 & 2.7e-07 & 2.4e-07 & 1.7e-07 & 1.6e-07 & 1.6e-07\tabularnewline
2 & 12 & 7.9e-11 & 7.3e-11 & 3.6e-11 & 1.5e-11 & 3.6e-11\tabularnewline
2 & 15 & 6.2e-14 & 1.5e-14 & 8.4e-15 & 5.1e-15 & 8.7e-15\tabularnewline
2 & 18 & 4.1e-13 & 1.3e-14 & 1.6e-14 & 6.0e-15 & 4.9e-15\tabularnewline
3 & 6 & 1.1e-03 & 8.4e-04 & 8.1e-04 & - & 7.8e-04\tabularnewline
3 & 9 & 9.5e-07 & 1.6e-06 & 1.3e-06 & - & 1.1e-06\tabularnewline
3 & 12 & 4.0e-10 & 1.1e-09 & 7.4e-10 & - & 4.6e-10\tabularnewline
3 & 15 & 1.2e-13 & 3.6e-13 & 1.8e-13 & - & 9.0e-14\tabularnewline
3 & 18 & 6.3e-13 & 9.9e-14 & - & - & 4.6e-14\tabularnewline
\bottomrule
\end{longtable}

\begin{longtable}[]{@{}rrrrrrr@{}}
\caption{Interpolation errors \(\|I_X(f_B) - f_B\|_\infty\)\label{tab:interperr_runge}}\tabularnewline
\toprule
\(d\) & \(n\) & equispaced & BLP & warp \& blend & RSV-lebgls & \({\RR}^d_{\XX}\)\tabularnewline
\midrule
\endfirsthead
\toprule
\(d\) & \(n\) & equispaced & BLP & warp \& blend & RSV-lebgls & \({\RR}^d_{\XX}\)\tabularnewline
\midrule
\endhead
2 & 6 & 4.5e-01 & 3.0e-01 & 3.1e-01 & 3.0e-01 & 3.1e-01\tabularnewline
2 & 9 & 6.6e-01 & 2.4e-01 & 1.7e-01 & 1.7e-01 & 1.7e-01\tabularnewline
2 & 12 & 1.1e+00 & 2.6e-01 & 9.8e-02 & 7.9e-02 & 9.9e-02\tabularnewline
2 & 15 & 1.9e+00 & 3.0e-01 & 6.2e-02 & 4.4e-02 & 6.8e-02\tabularnewline
2 & 18 & 3.1e+00 & 3.5e-01 & 2.7e-01 & 2.3e-02 & 4.9e-02\tabularnewline
3 & 6 & 6.5e-01 & 6.9e-01 & 7.1e-01 & - & 7.4e-01\tabularnewline
3 & 9 & 4.1e-01 & 4.9e-01 & 5.1e-01 & - & 5.6e-01\tabularnewline
3 & 12 & 1.0e+00 & 1.6e+00 & 7.7e-01 & - & 2.3e-01\tabularnewline
3 & 15 & 1.9e+00 & 2.4e+00 & 9.0e-01 & - & 1.4e-01\tabularnewline
3 & 18 & 4.5e+00 & 4.3e+00 & - & - & 1.3e-01\tabularnewline
\bottomrule
\end{longtable}

The relative sizes of the interpolation errors in Tables
\ref{tab:interperr_smooth} and \ref{tab:interperr_runge} more or less
correspond to the relative sizes of the Lebesgue constants. For the more
difficult function \(f_B\) the recursive nodes \(\RR^3_{\XX}\) continue to converge
at \(n=18\) when the other node sets have already begun to diverge.

\hypertarget{conclusion}{%
\section{Conclusion}\label{conclusion}}

How and by whom should the nodes \(\RR^d_{\XX}\) defined by the recursive rule
\eqref{eq:recursive} be used? The comparisons in this paper have made the case
that is the best explicit construction rule thus far, because of its
simplicity, its symmetry, and its performance in the metrics that matter to
finite element construction (but not for producing asymptotically convergent
interpolants). It does not outperform the Warburton's \emph{warp \& blend} node
family in 2D, so software already using those would not benefit from switching,
but its performance is superior to all other explicit node families in 3D,
particularly for \(n \geq 7\). Likewise, where implicitly defined node
families---such as Rapetti, Sommariva, and Vianello's LEBGLS nodes---have been
computed and published, they are superior to the \(\RR^d_{\XX}\) node family,
especially in 2D for \(n \geq 10\). But at the time of this writing the
tetrahedron has not received nearly as much attention as the triangle, and so
this new node family is the best available in 3D.

In Section \ref{symmetry} it was argued that the edge trace property was
useful in aligning nodes with tensor-product cells in hybrid meshes. In 3D, a
further necessary condition is for the traces on triangular facets to align
with neighboring nodes pyramid cells. There is no perfect analogue for the
overdetermined projections \eqref{eq:overdetermined} in the pyramid, so the
development of a matching node construction for the pyramid is left for future
work, where it would have to be compared against existing node sets such as
those developed by \textcite{ChWa15}.

\hypertarget{acknowledgements}{%
\section{Acknowledgements}\label{acknowledgements}}

The methods in this work were discovered in the course of research funded by
grant \#DE-SC0016140 from the U.S.~Department of Energy's Office of Advanced
Scientific Research. The author also thanks the anonymous referees who
suggested additional references and improvements to this work.

\printbibliography[title=References]

\end{document}